\documentclass[11pt]{article}

\usepackage{amsmath}
\usepackage{amsthm}
\usepackage{amssymb}
\usepackage{amsfonts}
\usepackage{forest}
\usepackage{mathrsfs}
\usepackage{bbm}
\usepackage{bbold}
\usepackage{setspace}
\usepackage{thmtools}
\usepackage{thm-restate}
\usepackage{fullpage}
\usepackage{tcolorbox}
\usepackage[all]{xy}
\usepackage[title, titletoc]{appendix}
\usepackage[ruled,linesnumbered]{algorithm2e}
\usepackage{mathalfa}
\usepackage{amsbsy}
\usepackage{comment}
\usepackage{xr}

\newtheorem{theorem}{Theorem}[section]
\newtheorem{definition}[theorem]{Definition}
\newtheorem{lemma}[theorem]{Lemma}

\newtheorem{corollary}[theorem]{Corollary}
\newtheorem{observation}[theorem]{Observation}

\newtheorem*{notation*}{Notation}
\newtheorem*{lemma*}{Lemma}
\newtheorem*{proposition*}{Proposition}
\newtheorem*{note*}{Note}

\DeclareMathOperator*{\argmin}{arg\,min}
\DeclareMathOperator*{\sign}{sign}

\DeclareMathOperator*{\im}{Im}

\newcommand{\parens}[1]{\left( #1 \right)}
\newcommand{\sparens}[1]{\left[ #1 \right]}
\newcommand{\leftsparens}[1]{\left[ #1 \right.}
\newcommand{\rightsparens}[1]{\left. #1 \right]}
\newcommand{\innerprod}[1]{\left< #1 \right>}

\newcommand{\abs}[1]{\left| #1 \right|}
\newcommand{\norm}[1]{\left\| #1 \right\|}

\newcommand{\set}[1]{\left\{ #1 \right\}}
\newcommand{\suchthat}[0]{\middle|}

\DeclareMathOperator{\ex}{\mathbbm{E}}
\newcommand{\ev}[2]{\ex_{#1}{\sparens{#2}}}
\NewDocumentCommand{\evtl}{m m m O{\\ &}}{\ex_{#1}{\leftsparens{\vphantom{#3} #2} \phantom{\ex_{#1}}} #4 \qquad \qquad \rightsparens{\vphantom{#2} #3}}

\DeclareMathOperator{\weightletter}{w}
\newcommand{\skeleton}[2]{#1^{\parens{#2}}}
\NewDocumentCommand{\weight}{O{k} o}{\IfNoValueTF{#2}{\weightletter \parens{ #1 }}{\weightletter_{#1}\parens{#2}}}

\newcommand{\stdcomplex}{X}
\newcommand{\dimension}[1]{\dim{\parens{#1}}}

\newcommand{\face}[1]{\MakeLowercase{#1}}
\newcommand{\vertex}[1]{\MakeLowercase{#1}}
\newcommand{\stdvertex}{\vertex{v}}
\newcommand{\genvertex}{\vertex{u}}
\newcommand{\stdface}{\face{\sigma}}
\newcommand{\genface}{\face{\tau}}

\newcommand{\coboundaryoperator}{\delta}
\newcommand{\cochain}[1]{\MakeUppercase{#1}}
\newcommand{\cochainset}[2]{C^{#1}\parens{#2}}
\newcommand{\levelcochainset}[3]{C_{#1}^{#2}\parens{#3}}
\newcommand{\stdcochain}{\cochain{f}}
\newcommand{\gencochain}{\cochain{g}}

\newcommand{\coboundaryset}[2]{B^{#1}\parens{#2}}

\usepackage[super]{nth}
\usepackage{algpseudocode}
\usepackage[T1]{fontenc}
\usepackage{hyperref}
\usepackage{xparse}

\NewDocumentCommand{\upoperator}{O{k} o}{\IfNoValueTF{#2}{M^{+}_{#1}}{M^{+{#1}}_{#2}}}
\NewDocumentCommand{\nlupoperator}{O{k}}{\parens{M'}^{+}_{#1}}
\NewDocumentCommand{\downoperator}{O{k} o}{\IfNoValueTF{#2}{M^{-}_{#1}}{M^{-{#1}}_{#2}}}
\NewDocumentCommand{\walkoperator}{O{k} o}{\IfNoValueTF{#2}{\Delta_{#1}}{\Delta_{#1, #2}}}

\newcommand{\viewerletter}{\Lambda}
\NewDocumentCommand{\viewer}{o o}{\IfNoValueTF{#1}{\viewerletter}{\IfNoValueTF{#2}{\viewerletter_{#1}}{\viewerletter_{#1}#2}}}
\NewDocumentCommand{\locviewer}{o o}{\IfNoValueTF{#1}{\viewerletter^{\ell}}{\IfNoValueTF{#2}{\viewerletter^{\ell}_{#1}}{\viewerletter^{\ell}_{#1}#2}}}
\NewDocumentCommand{\resviewer}{o o}{\IfNoValueTF{#1}{\viewerletter^r}{\IfNoValueTF{#2}{\viewerletter^r_{#1}}{\viewerletter^r_{#1}#2}}}
\NewDocumentCommand{\dimdiff}{O{\viewer}}{\Delta\parens{#1}}

\newcommand{\sldoperator}{d}
\DeclareMathOperator{\image}{Im}

\newcommand{\R}{\mathbbm{R}}
\newcommand{\ind}[1]{\mathbbm{1}_{#1}}

\SetKwProg{Pick}{pick}{}{}

\title{Fine Grained Analysis of High Dimensional Random Walks}
\author{
Roy Gotlib
\footnote{Department of Computer Science, Bar-Ilan University, roy.gotlib@gmail.com, research supported by ERC and ISF. } 
\and 
Tali Kaufman
\footnote{Department of Computer Science, Bar-Ilan University, kaufmant@mit.edu, research supported by ERC and ISF.}
}

\begin{document}
    \maketitle
    \begin{abstract}
        One of the most important properties of high dimensional expanders is that high dimensional random walks converge rapidly.
        This property has proven to be extremely useful in a variety of fields in the theory of computer science from agreement testing to sampling, coding theory and more.
        In this paper we present a state of the art result in a line of works analyzing the convergence of high dimensional random walks~\cite{DBLP:conf/innovations/KaufmanM17,DBLP:conf/focs/DinurK17, DBLP:conf/approx/KaufmanO18,DBLP:journals/corr/abs-2001-02827}, by presenting a \emph{structured} version of the result of~\cite{DBLP:journals/corr/abs-2001-02827}.
        While previous works examined the expansion in the viewpoint of the worst possible eigenvalue, in this work we relate the expansion of a function to the entire spectrum of the random walk operator using the structure of the function; We call such a theorem a Fine Grained High Order Random Walk Theorem.
        In sufficiently structured cases the fine grained result that we present here can be much better than the worst case while in the worst case our result is equivalent to~\cite{DBLP:journals/corr/abs-2001-02827}.

        In order to prove the Fine Grained High Order Random Walk Theorem we introduce a way to bootstrap the expansion of random walks on the vertices of a complex into a fine grained understanding of higher order random walks, provided that the expansion is good enough.
        
        In addition, our \emph{single} bootstrapping theorem can simultaneously yield our Fine Grained High Order Random Walk Theorem as well as the well known Trickling down Theorem. Prior to this work, High order Random walks theorems and Tricking down Theorem have been obtained from different proof methods.
    \end{abstract}
    
\section{Introduction}\label{sec:introduction}
In recent years much attention has been given to the field of high dimensional expanders which are high dimensioanl analogues of expander graphs.
One extremely useful property of high dimensional expanders is that higher dimensional random walks (which are higher dimensional analogues of random walks on graphs) converge rapidly to their stationary distribution (For example, this property was used in~\cite{DBLP:conf/focs/DinurK17, DBLP:conf/innovations/KaufmanM20, https://doi.org/10.48550/arxiv.1811.01816, DBLP:conf/soda/DinurHKNT19} and more).
Consequently there has been some work studying the convergence of higher dimensional random walks~\cite{DBLP:conf/focs/DinurK17,DBLP:conf/innovations/KaufmanM17,DBLP:conf/approx/KaufmanO18,DBLP:journals/corr/abs-2001-02827}.
In this paper we improve upon these convergence results by relating the structure of the function to its expansion.
Specifically we present the following improvements:

\paragraph{Fine grained analysis of random walk} Prior to this paper, the state of the art analysis of high dimensional random walks was done by Alev and Lau in~\cite{DBLP:journals/corr/abs-2001-02827} following~\cite{DBLP:conf/focs/DinurK17,DBLP:conf/innovations/KaufmanM17,DBLP:conf/approx/KaufmanO18}.
Their work analyzed the eigenvalues of an important random walk called the \emph{down-up random walk}.
Their result, however, was only useful for the worst case analysis as it did not relate the structure of the function to its expansion and thus was forced to consider the worst possible function.
In this paper we present an improvement upon Alev and Lau's result by finding a connection between the structure of a function and how well it expands and can therefore yield better results on cochains that posses a ``nice'' structure.

In two-sided spectral expanders a fine grained analysis of high dimensional random walks was already proven, based on Fourier analysis in high dimensional expanders ~\cite{DBLP:conf/approx/DiksteinDFH18, DBLP:conf/approx/KaufmanO18}. In the two sided case even stronger results following from hypercontractivity are known~\cite{eigenstripping-pseudorandomness-and-unique-games, https://doi.org/10.48550/arxiv.2111.09444, https://doi.org/10.48550/arxiv.2111.09375} and more.
Our result is, importantly, about \emph{one-sided} spectral expanders as there are cases when the use of one-sided high dimensional expansion is crucial - for example in Anari et al's breakthrough proof of the fast convergence of the basis exchange walk~\cite{https://doi.org/10.48550/arxiv.1811.01816} and thus showed an algorithm that samples a basis of a matroid. 
This result started a wave of sampling results that use high dimensional expanders~\cite{chen2021optimal, anari2021spectral, anari2022entropic} to name a few examples.
This result relies heavily on fast convergence of high dimensional random walks on one-sided expanders (As they show that that the basis exchange corresponds to a down-up walk on the top dimension of a one-sided high dimensional expander).
In this work we present, to our knowledge, the first result to show a fine grained analysis of the random walk operator in one-sided local spectral expanders.

\paragraph{Replacing eigendecomposition} Previous fine grained analysis of high dimensional random walks relied on finding approximate eigendecomposition of the high dimensional random walk.
We present a new approach to finding a fine grained understanding of the high dimensioanl random walks: bootstrapping an understanding of the expansion of random walks on the vertices of the complex.
We show that, if the expansion of these random walks beat the expansion of high dimensional random walks on the vertices of local structures\footnote{Specifically, high dimensional random walks on the vertices of the links of vertices.}, we can bootstrap it into a \emph{fine grained} understanding of higher dimensional random walks.

\paragraph{Unification of the High order random walk theorem and the Tricking down theorem} In order to perform our fine grained analysis of the higher dimensional random walk operators we develop a new bootstrapping framework.
This new framework is fairly generic and seems to be of independent interest as it can be used to prove another central theorem in the theory of high dimensional expansion, namely the trickling down theorem~\cite{https://doi.org/10.48550/arxiv.1709.04431}. We comment that prior to this work these two important theorems were obtained by different proof techniques.

Before we can state our results more formally, we have to define the high dimensional analogs of expander graphs as graphs do not posses high dimensions.
This high dimensional object is called a ``simplicial complex'' and is defined as:
\begin{definition}[Simplicial complex]
    A set $\stdcomplex$ is a \emph{simplicial complex} if it is closed downwards meaning that if $\stdface \in \stdcomplex$ and $\genface \subseteq \stdface$ then $\genface \in \stdcomplex$.
    We call members of $\stdcomplex$ the \emph{faces} of $\stdcomplex$.
\end{definition}
Simplicial complexes can be thought of as hyper-graphs with closure property (i.e.\ every subset of a hyper-edge is a hyper-edge).
We are interested in higher dimensions and therefore it would be useful to define the dimension of these higher dimensional objects:
\begin{definition}[Dimension]
    Let $\stdcomplex$ be a simplicial complex and let $\stdface \in \stdcomplex$ be a face of $\stdcomplex$.
    Define the dimension of $\stdface$ to be:
    \[
        \dimension{\stdface} = \abs{\stdface}-1
    \]
    Denote the set of all faces of dimension $i$ in $\stdcomplex$ as $\stdcomplex(i)$.
    Also define the dimension of the complex $\stdcomplex$ as:
    \[
        \dimension{\stdcomplex} = \max_{\stdface \in \stdcomplex}{\set{\dimension{\stdface}}}
    \]
    Note that there is a single $(-1)$-dimensional face - the empty face.
\end{definition}
Of particular interest are simplicial complexes whose maximal faces are of the same dimension, defined below:
\begin{definition}[Pure simplicial complex]
    A simplicial complex $\stdcomplex$ is a \emph{pure simplicial complex} if every face $\stdface \in \stdcomplex$ is contained in some $\parens{\dimension{\stdcomplex}}$-dimensional face.
\end{definition}
Throughout this paper we will assume that every simplicial complex is pure.
In most cases we will be interested in weighted pure simplicial complexes.
In weighted pure simplicial complexes the top dimensional faces are weighted and the weight of the rest of the faces follows from there as described here:
\begin{definition}[Weight]
    Let $\stdcomplex$ be a pure $d$-dimensional simplicial complex.
    Define its weight function $\weightletter: \stdcomplex \rightarrow \sparens{0,1}$ to be a function such that:
    \begin{itemize}
        \item $\sum_{\stdface \in \stdcomplex(d)}{\weight[\stdcomplex]}=1$
        \item For every face $\genface$ of dimension $i<d$ it holds that:
        \[
            \weight[\genface] = \frac{1}{\binom{d+1}{i+1}}\sum_{\substack{\stdface \in \stdcomplex(d)\\\genface \subseteq \stdface}}{\weight[\stdface]}
        \]
    \end{itemize}
    It is important to note that we think of unweighted complex as complexes that satisfy $\forall \stdface \in \stdcomplex(d):\weight[\stdface] = \frac{1}{\abs{\stdcomplex(d)}}$.
    While the top dimensional faces of unweighted complexes all have the same weight, the same cannot be said for lower dimensional faces.
    
    It is also important to note that the sum of weights in every dimension is exactly $1$ and therefore for every $k$ the weight function can, and at times will, be thought of as a distribution on $\stdcomplex(k)$.
\end{definition}
One key property of high dimensional expanders is that they exhibit \emph{local to global} phenomena.
These phenomena are at the main interest of this paper.
It is therefore useful to consider local views of the simplicial complex which we define as follows:
\begin{definition}[Link]
    Let $\stdcomplex$ be a simplicial complex and $\stdface \in \stdcomplex$ be a face of $\stdcomplex$.
    Define the \emph{link} of $\stdface$ in $\stdcomplex$ as:
    \[
        \stdcomplex_{\stdface} = \set{\genface \setminus \stdface \suchthat \stdface \subseteq \genface}
    \]
    It is easy to see that the link of any face is a simplicial complex.

    The weight of faces in the links is induced by the weights of the faces in the original complex.
    Specifically, we denote by $\weightletter_{\stdface}$ the weight function in the link of $\stdface \in \stdcomplex(i)$ and it holds that:
    \[
        \forall \genface \in \stdcomplex_{\stdface}(j): \weight[\stdface][\genface] = \frac{\weight[\genface \cup \stdface]}{\binom{i+j+2}{i+1}\weight[\stdface]}
    \]
    Generally speaking, the local to global phenomena are ways to derive properties of the entire complex by only looking at local views (i.e.\ links).
\end{definition}
Another important substructure of a simplicial complex is its skeletons
\begin{definition}[Skeleton]
    Let $\parens{\stdcomplex, \weightletter}$ be a weighted pure $d$-dimensional simplicial complex and let $i \le d$.
    Define the $i$-skeleton of $\stdcomplex$ as the following weighted simplicial complex:
    \[
        \skeleton{\stdcomplex}{i} = \set{\stdface \in \stdcomplex \suchthat \dim{\stdface} \le i}
    \]
    With the original weight function.
\end{definition}
In many cases we will think of the $1$-skeleton of a simplicial complex as a graph.
In addition, it is important to note that even if the original complex is unweighted, the skeletons of said complex might still be weighted.

We are now ready to define a high dimensional expander\footnote{There is no singular definition of high dimensional expander but for the majority of this paper we only use the algebraic definition - local spectral expansion.}.
\begin{definition}[Local spectral expander]
    A pure $d$-dimensional simplicial complex $\stdcomplex$ is a $\lambda$-local spectral expander\footnote{Much like one dimensional expanders, in high dimensions there is also notion of one-sided vs. two-sided local spectral expansion. The definition we use throughout the paper is that of \emph{one}-sided local spectral expander. The difference being that in two-sided local spectral expander the underlying graph of every link is a two-sided expander rather than a one-sided expander.} if for every face $\stdface$ of dimension at most $d-2$ it holds that $\skeleton{\stdcomplex_{\stdface}}{1}$ is a $\lambda$-spectral expander\footnote{The complexes are weighted and therefore their expansion property is defined as the second largest eigenvalue of the non-lazy random walk. A random walk that walks from a face to one of its neighbours with probability equal to the proportion between the weight of the edge that connects them and the sum of the weights of the edges that include said vertex.}.
    Note that this includes $\stdface = \emptyset$, i.e.\ the entire complex.
\end{definition}

Much like graphs, simplicial complexes also support random walks.
In graphs, the random walks are of the form vertex-edge-vertex - the walk might move between two vertices if they are connected by an edge.
The high dimensional analogue of these random walks travel between two $k$-dimensional faces if they are part of a common $(k+1)$-dimensional face.
Our particular random walk of interest is the higher dimensional analogue of the non-lazy random walk, defined as follows:
\begin{definition}[Non-lazy up-down operator informal, for formal see~\ref{def:non-lazy-up-down-random-walk}]
    Define the $k$-dimensional non-lazy up-down random walk, $\nlupoperator[k]$ as the $k$ dimensional analogue of the non-lazy random walk on the vertices of a graph:
    A walk that moves between two $k$-dimensional faces if they are contained in a $(k+1)$-dimensional face and never stays in place.
\end{definition}
We are going to improve our understanding of how these higher dimensional random walks apply to structured states.
These states correspond to another natural structure on high dimensional expanders called cochains that is defined as follows:
\begin{definition}[Cochains]
    Let $\stdcomplex$ be a pure $d$-dimensional simplicial complex.
    For $-1 \le k \le d$ define a $k$-dimensional cochain $\stdcochain$ to be any function from $\stdcomplex(k)$ to $\R$.
    We also denote by $\cochainset{k}{\stdcomplex;\R}$ the set of all $k$-dimensional cochains.
\end{definition}
We are going to be interested in ways of viewing the cochains in the links of the complex.
For now we will only introduce one such way.
Namely localization:
\begin{definition}[Localization]
    Let $\stdcomplex$ be a pure $d$-dimensional simplicial complex, $k,i$ be dimensions such that $i < k$ and $\stdcochain \in \cochainset{k}{\stdcomplex;\R}$.
    Also let $\stdface \in \stdcomplex(i)$.
    Define the localization of $\stdcochain$ to $\stdface$ to be:
    \[
        \stdcochain_{\stdface}(\genface) = \stdcochain(\stdface \cup \genface)
    \]
\end{definition}

We note that there is a very natural inner product defined on the cochains of a simplicial complex, defined as follows:
\begin{definition}[Inner product]
    Let $\stdcomplex$ be a pure $d$-dimensional simplicial complex and let $\stdcochain, \gencochain \in \cochainset{k}{\stdcomplex;\R}$.
    Define the inner product of $\stdcochain$ and $\gencochain$ to be:
    \[
        \innerprod{\stdcochain, \gencochain} = \sum_{\stdface \in \stdcomplex(k)}{\weight[\stdface]\stdcochain(\stdface)\gencochain(\stdface)}
    \]
\end{definition}
\subsection{Main Results}\label{subsec:main-results}
In this paper we present a state of the art analysis the high dimensional analogue of the non-lazy random walk.
We are specifically interested in going beyond Alev and Lau's worst case result~\cite[Theorem 1.5]{DBLP:journals/corr/abs-2001-02827} and relate the structure of the cochain to its expansion.
Specifically, we define:
\begin{definition}[$i$-level cochain, informal. For formal see Definition~\ref{def:i-level-cochain}]
    A cochain $\stdcochain$ is an $i$-level cochain with respect to localization if for every $\stdface \in \stdcomplex(i-1)$ it holds that $\innerprod{\stdcochain_{\stdface}, \ind{}}=0$.\footnote{We note that every $j$-dimensional $i$-level cochain corresponds to a cochain in the $i$\textsuperscript{th} dimension that has been ``lifted'' up to the $j$\textsuperscript{th} dimension. A more formal version of this statement can be found in Lemma~\ref{lem:random-walk-advantage}.}
\end{definition}
Previously the best analysis of high dimensional random walk was due to Alev and Lau~\cite[Theorem 1.5]{DBLP:journals/corr/abs-2001-02827} who showed that:
\begin{theorem}[Alev and Lau~\cite{DBLP:journals/corr/abs-2001-02827}, restated]
    Let $\stdcomplex$ be a pure $d$-dimensioanl high dimensional expander.
    Define $\gamma_i = \max_{\stdface \in \stdcomplex(i)}{\set{\lambda_2(\stdcomplex_{\stdface})}}$.
    For any dimension $k$ and any $k$-dimensional $0$-level cochain $\stdcochain$ it holds that:
    \[
        \innerprod{\nlupoperator[k]\stdcochain, \stdcochain}
        \le \parens{1-\frac{1}{k+1}\prod_{j=-1}^{k-2}{(1-\gamma_i)}} \norm{\stdcochain}^2
    \]
\end{theorem}
We show that this result can be vastly improved for structured cochains.
Specifically, we show the following decomposition of the high dimensional random walks:
\begin{theorem}[Fine grained analysis of high dimensional random walks, Informal. For formal see Theorem~\ref{thm:walk-operator-decomposition}]
    Let $\stdcochain \in \cochainset{k}{\stdcomplex;\R}$ and let $\stdcochain_0,\cdots,\stdcochain_k$ be an orthogonal decomposition of $\stdcochain$ such that $\stdcochain_i$ is an $i$-level cochain and $\stdcochain = \sum_{i=0}^k{\stdcochain_i}$.
    In addition, let $\gamma_i = \max_{\stdface \in \stdcomplex(i)}{\set{\lambda_2(\stdcomplex_{\stdface})}}$ (Where $\lambda_2(\stdcomplex_{\stdface})$ is the second largest eigenvalue of the underlying graph of $\stdcomplex_{\stdface}$) then:
    \[
        \innerprod{\nlupoperator[k]\stdcochain, \stdcochain}
        \le \sum_{i = 0}^{k}{\parens{1-\frac{1}{k-i+1}\prod_{j=i-1}^{k-1}{(1-\gamma_{j})}}\norm{\stdcochain_i}^2}
    \]
\end{theorem}
\paragraph{Cases where we improve upon previous results} In the worst case (i.e. no assumption is made on the structure of $\stdcochain$) our result matches that of Alev and Lau.
In cases where the cochain is structured, however, our Theorem yields strictly better results than what was previously known.
We also give some examples of families of structured cochains on which our result is \emph{strictly better} than the result of Alev and Lau - Specifically, we show two families of cochains that are highly structured:
The first is a set of cochains associated with a different form of high dimensional expansion, the minimal cochains and the second is the indicator function of a balanced set of faces (for more information, see Subsection~\ref{subsec:examples-of-cochains-of-high-level}).

\paragraph{Comparison with other known decompositions} Similar decompositions of the high dimensional random walks were already known for two-sided high dimensional expanders~\cite{DBLP:conf/approx/DiksteinDFH18, DBLP:conf/approx/KaufmanO18}. These decomposition relied on finding approximate eigenspaces of the walk operator.
Unlike the case in two-sided high dimensional expanders, in one-sided high dimensional expanders the eigenspaces of the high dimensional random walks are not currently understood, even approximately.
We do, however, understand the expansion of key random walks on the $0$-dimensional cochains.
We follow this by showing that if this expansion is strong enough (i.e.\ these random walks converge fast enough) we can boost it to all levels and get a decomposition theorem without eigendecomposition.
In order to apply our bootstrapping method we have to show that the expansion of the aforementioned $0$-dimensional cochains ``beats'' the expansion of cochains of higher levels. 
Note that the decomposition achieved by our bootstrapping theorem differs from the decomposition known in two-sided expanders in one \emph{crucial} way: In our decomposition the level functions are \emph{not} approximate eigenfunctions.
Moreover, applying the non-lazy up-down random walk operator to any one of them yields a cochain that is \emph{not} orthogonal to many of the other $\stdcochain_j$s.
This allows us to sidestep a major technical barrier in previous works as we do not rely on the existence of an eigendecomposition of the walk operator (or even an approximate decomposition of that operator).
Note that many results that use local spectral expanders, such as their usage in proving the existence of cosystolic expanders, agreement testing, locally testable codes and more only require one application of the walk operator and thus we believe that our result will prove to be influential with regards to these fields.

Our main tool for proving the decomposition of high dimensional random walks is a bootstrapping theorem that reduces decomposition of higher dimensional random walk to an understanding of highly structured cochains (for example cochains that correlate with $0$-dimensional cochains).
The bootstrapping theorem is fairly general and thus we bring a special case of it here:
\begin{theorem}[Bootstrapping theorem, informal. For formal see Theorem~\ref{thm:main}]\label{thm:main-informal}
    Let $\stdcomplex$ be a simplicial complex and $k$ be a dimension.
    Every $\stdcochain \in \cochainset{k}{\stdcomplex;\R}$ that is orthogonal to the constant functions can be decomposed into $\stdcochain = \sum_{i=0}^k{\stdcochain_i}$ such that the cochains $\stdcochain_i$ are both:
    \begin{enumerate}
        \item \emph{$i$-level cochains with respect to localization}.
        \item \emph{Orthogonal:} For every $i \ne j$ it holds that $\stdcochain_i$ is orthogonal to $\stdcochain_j$.
    \end{enumerate}
    We can use a solution of some recursive formula on $0$-level cochains and values $\set{\lambda_i}_{i=0}^k$ in order to bootstrap a decomposition of the following form:
    \[
        \innerprod{\nlupoperator[k]\stdcochain, \stdcochain} = \sum_{i=0}^k{\lambda_i \norm{\stdcochain_i}^2}
    \]
\end{theorem}

Solving the recursive formula in the bootstrapping theorem requires us to gain some ``advantage'' for highly structured cochains in the complex.
We can therefore view this Theorem as a tool that allows us to bootstrap an advantage we have to a decomposition of the non-lazy random walks.

As we said, our bootstrapping theorem is fairly generic and can also yield the celebrated Oppenheim's trickling down theorem~\cite[Theorem~4.1]{https://doi.org/10.48550/arxiv.1709.04431}:
\begin{theorem}[Trickling Down, \cite{https://doi.org/10.48550/arxiv.1709.04431}]\label{thm:trickling-down}
    Let $\stdcomplex$ be a pure $d$-dimensional simplicial complex.
    If it holds that:
    \begin{itemize}
        \item For every vertex $\stdvertex$: $\skeleton{\stdcomplex_{\stdvertex}}{1}$ is a $\lambda$ spectral expander.
        \item $\stdcomplex$ is connected.
    \end{itemize}
    Then it holds that $\skeleton{\stdcomplex}{1}$ is a $\frac{\lambda}{1-\lambda}$ spectral expander.
\end{theorem}
In order to prove the trickling down we use Theorem~\ref{thm:main-informal} while defining the $i$-level cochains differently.
For more detail, see Section~\ref{sec:trickling-down}.
\subsection{Proof Layout}\label{subsec:proof-layout}
As we mentioned before, we analyze the \emph{non-lazy up-down random walk} - a high dimensional analogue of the non-lazy random walk in graphs.
We would like to get a decomposition of the non-lazy random walk operator.
In order to do that we decompose the space of cochains to spaces which we term level cochains.
Our proof is then comprised of two steps:
A bootstrapping Lemma that reduces the problem of decomposing the non-lazy up-down random walk operator to a simple recursive condition about the $0$-level cochains and an advantage that allows us to solve said recursive condition.
We note that solving the recursive condition is considerably simpler than proving the decomposition directly as the $0$-level cochains are very structured objects.

\paragraph{Link viewers} Local-to-global arguments are typically structured by taking a global cochain, decomposing it to local cochains (i.e.\ cochains on the links of the vertices of the complex) and then using a property of the cochains in the links in order to argue about the original cochain.
Two extremely useful ways of decomposing a global cochain to cochains in the links are \emph{restriction} and \emph{localization}.
These two methods share many properties, for example:
\begin{itemize}
    \item They preserve constant functions - Any global constant function is constant locally.
    \item They are linear - The local view of sum of any two global cochains is the sum of the local views.
    \item They interact well with the inner product - The expected value of the local inner products of two cochains is their inner product.
\end{itemize}
We start by identifying these properties and defining a general object that satisfy them called a \emph{link viewer}.
A link viewer $\viewer$ defines the local view of a cochain $\stdcochain$ in the link of $\stdface$ which we denote by $\viewer[\stdface][\stdcochain]$.
For the rest of this section, unless stated otherwise, we will think of the \emph{localization link viewer} defined as $\locviewer[\stdface][\stdcochain] = \stdcochain_{\stdface}$.

\paragraph{Level cochains} Any link viewer decomposes the space of cochains to subspaces called \emph{$i$-level cochains}.
These spaces are the set of cochains whose expected value is $0$ when viewed from any $i$-dimensional face.
We think of these cochains as not correlated with any $i$-dimensional face.
For example, under the localization link viewer the $i$-level cochains are cochains $\gencochain$ that satisfy: $\forall \stdface \in \stdcomplex(i): \ev{\genface \in \stdcomplex_{\stdface}(k-i-1)}{\locviewer[\stdface][\gencochain](\genface)} = \ev{\genface \in \stdcomplex_{\stdface}(k-i-1)}{\gencochain_{\stdface}(\genface)} = 0$.
Note that any $i$-level cochain is also an $(i-1)$-level cochain and that of $\stdcochain$ is an $i$-level cochain then $\viewer[\stdvertex][\stdcochain]$ is an $(i-1)$-level cochain in the link of every vertex $\stdvertex$.
It is therefore useful to define \emph{proper level cochains} as well - a proper $i$-level cochain is an $i$-level cochain that is orthogonal to every $(i-1)$-level cohain.
Every cochain $\stdcochain$ can be decomposed to $\stdcochain = \sum_{i=-1}^p{\stdcochain_i}$ such that every $\stdcochain_i$ is a proper $i$-level cochain.
This decomposition plays a key role in our proof.

\paragraph{The bootstrapping argument} We can now present our proof for the bootstrapping argument.
We assume that the non-lazy up-down random walk expands locally and show that this yields that the non-lazy up-down random walk expands globally.
We restrict our attention only to $0$-level cochains (which are not necessarily proper) as $(-1)$-level cochains correspond to the trivial eigenvalue in which we are not interested.
Let $\stdcochain=\sum_{i=0}^p{\stdcochain_i}$ be a $0$-level cochain.
Consider its localization to the links of vertices $\viewer[\stdvertex][\stdcochain]=\sum_{i=0}^p{\viewer[\stdvertex][\stdcochain_i]}$.
Note that $\sum_{i=1}^p{\viewer[\stdvertex][\stdcochain_i]}$ is a $0$-level cochain in $\stdcomplex_{\stdvertex}$(as the localization to links of vertices of $i$-level cochains is an $(i-1)$-level cochain).
Note that $\viewer[\stdvertex][\stdcochain_0]$ is \emph{not} a $0$-level cochain.
We can, however, apply our decomposition theorem to some part of $\viewer[\stdvertex][\stdcochain_0]$ by decomposing it to two parts:
\begin{itemize}
    \item A $0$-level cochain - $\viewer[\stdvertex][\stdcochain_0] - \ev{}{\viewer[\stdvertex][\stdcochain_0]}$
    \item A remainder - $\ev{}{\viewer[\stdvertex][\stdcochain_0]}$.
\end{itemize}
We can therefore apply our decomposition theorem to $\sum_{i=0}^p{\viewer[\stdvertex][\stdcochain_i]}-\ev{}{\viewer[\stdvertex][\stdcochain_0]}$.
Note that a key step in this decomposition is that despite the fact that these new localized cochains are \emph{not} orthogonal to each other, they are \emph{all} orthogonal to the constants.
This allows us to separate them from the constant part of every localization.
All we have left to do is to compensate for the remainder.
This is done using the advantage step which we will describe next.

\paragraph{The Advantage} In the advantage step we try to bound the remainder we have left in the bootstrapping stage.
The advantage required in order to achieve the random walk decomposition Theorem is done by considering the localization link viewer that showing the following observations:
\begin{enumerate}
    \item Any $0$-level function $\stdcochain_0$ is, in a sense, a high dimensional description of some other $0$-dimensional cochain $\gencochain \in \cochainset{0}{\stdcomplex;\R}$.
    This is true in the sense that $\norm{\stdcochain_0}=\norm{\gencochain}$ and $\ev{}{\locviewer[\stdvertex][\stdcochain_0]} = \norm{\upoperator[k][0]\gencochain}$ where $\upoperator[k][0]$ is the random walk the applied the up step $k$ times followed by applying the down step $k$ times.

    \item There is a connection between the random walk on the underlying graph of the complex and $\upoperator[k][]$.
    The key claim we use is that once the random walk has performed its first up step \emph{it can already ``see'' all the vertices it will ever see} (this is due to the structure of simplicial complexes - if two vertices share a $k$-dimensional face then by the closure property they also share an edge).
    Going further up only decreases the probability of staying in place.
    Therefore $\upoperator[k][]$ is a weighted sum of the underlying graph's non-lazy random walk transition matrix and a lazy component (which corresponds to staying in place).
    We note that this observation can be generalized to any random walk on $k$-dimensional faces that ``walks through'' faces of dimension higher than $2k$.
\end{enumerate}
Using these observations as well as our understanding of the $0$-dimensional random walk we obtain the advantage we seek.

\section{The Signless Differential and Its Adjoint Operator}\label{sec:the-signless-differential-and-its-adjoint-operator}
One of the key operators induced by any simplicial complex is its \emph{signless differential operator}.
The signless differential operator is an averaging operator that accepts a $k$-dimensional cochain and returns a $(k+1)$-dimensional cochain.
We adopt the terminology of~\cite{DBLP:conf/approx/KaufmanO18} and consider repeated application of the signless differential and its adjoint operator.

\begin{definition}[Signless differential]
    The signless differential operator \\ $\sldoperator_k: \cochainset{k}{\stdcomplex; \R} \rightarrow \cochainset{k+1}{\stdcomplex; \R}$ in the following way:
    \[
        \sldoperator_k \stdcochain(\stdface) 
        = \ev{\genface \in \stdcomplex(k)}{\stdcochain(\genface) \suchthat \genface \subseteq \stdface}
        = \sum_{\genface \in \binom{\stdface}{k+1}}{\frac{1}{k+2}\stdcochain\parens{\genface}}
    \]
    When the dimension is clear from context it will be omitted from the notation.
    Also define \\ $\sldoperator_k^*: \cochainset{k+1}{\stdcomplex ; \R} \rightarrow \cochainset{k}{\stdcomplex ; \R}$ to be the adjoint operator of $\sldoperator_k$.
\end{definition}
\begin{note*}
    The signless differential does not meet the definition of a differential as $\sldoperator_k \sldoperator_{k-1} \ne 0$.
\end{note*}
\begin{lemma}[{\cite[Lemma~3.6]{DBLP:conf/approx/KaufmanO18}}, \cite{DBLP:conf/approx/DiksteinDFH18}]\label{lem:sld-adj-explicit}
    It holds that:
    \[
        \sldoperator_{k}^* \stdcochain(\genface) 
        = \ev{\stdface \in \stdcomplex(k+1)}{\stdcochain(\stdface) \suchthat \genface \subseteq \stdface} = \ev{\stdface \in \stdcomplex_{\genface}(0)}{\stdcochain_{\genface}(\stdface)}
    \]
\end{lemma}
\begin{proof}
    Consider, now, the following:
    \begin{align*}
        \innerprod{\sldoperator \stdcochain, \gencochain} & =
        \sum_{\stdface \in \stdcomplex(k+1)}{\weight[\stdface] \sldoperator\stdcochain(\stdface) \gencochain(\stdface)} =
        \sum_{\stdface \in \stdcomplex(k+1)}{\weight[\stdface] \ev{\genface \in \stdcomplex(k)}{\stdcochain(\genface) \suchthat \genface \subseteq \stdface} \gencochain(\stdface)} = \\
        & = \sum_{\stdface \in \stdcomplex(k+1)}{\weight[\stdface] \parens{\sum_{\substack{\genface \in \stdcomplex(k) \\ \genface \subseteq \stdface}}{\frac{1}{k+2}\stdcochain(\genface)}}\gencochain(\stdface)} = \\
        & = \sum_{\stdface \in \stdcomplex(k+1)}{\sum_{\substack{\genface \in \stdcomplex(k) \\ \genface \subseteq \stdface}}{\frac{\weight[\stdface]}{k+2}\stdcochain(\genface)\gencochain(\stdface)}} = \\
        & = \sum_{\genface \in \stdcomplex(k)}{\sum_{\substack{\stdface \in \stdcomplex(k+1) \\ \genface \subseteq \stdface}}{\frac{\weight[\stdface]}{(k+2) \weight[\genface]}\weight[\genface]\stdcochain(\genface)\gencochain(\stdface)}} = \\
        & = \sum_{\genface \in \stdcomplex(k)}{\weight[\genface]\stdcochain(\genface)\sum_{\substack{\stdface \in \stdcomplex(k+1) \\ \genface \subseteq \stdface}}{\frac{\weight[\stdface]}{(k+2) \weight[\genface]}\gencochain(\stdface)}} = \\
        & = \sum_{\genface \in \stdcomplex(k)}{\weight[\genface]\stdcochain(\genface)\sum_{\substack{\stdface \in \stdcomplex(k+1) \\ \genface \subseteq \stdface}}{\weight[\genface][\stdface]\gencochain(\stdface)}} = \\
        & = \sum_{\genface \in \stdcomplex(k)}{\weight[\genface]\stdcochain(\genface)\ev{\stdface \in \stdcomplex(k+1)}{\gencochain(\stdface) \suchthat \genface \subseteq \stdface}}
    \end{align*}
\end{proof}
We will be interested in repeated application of the signless differential and its adjoint operator.
To that effect it will be useful to present them explicitly.
Lemmas~\ref{lem:multi-sld-operator} and~\ref{lem:multi-sld-operator-adj} will present repeated applications of these operators explicitly.
\begin{lemma}\label{lem:multi-sld-operator}
Let $\stdcochain$ be an $i$-dimensional cochain.
Then:
\[
    \sldoperator_{j-1}\cdots \sldoperator_{i}\stdcochain(\stdface) = \ev{\genface \in \stdcomplex\parens{i}}{\stdcochain(\genface) \suchthat \genface \subseteq \stdface}
\]
\end{lemma}
\begin{proof}
    By induction, the case of $j=i+1$ follows from the definition of $\sldoperator_i$.

    Induction step:
    \begin{align*}
        \sldoperator_{j}\cdots \sldoperator_{i}\stdcochain(\stdface)
        & = \ev{\genface \in \stdcomplex(j)}{\sldoperator_{j}\cdots \sldoperator_{i}\stdcochain(\genface) \suchthat \genface \subseteq \stdface}
        = \ev{\genface \in \stdcomplex(j)}{\ev{\genface' \in \stdcomplex\parens{i}}{\stdcochain(\genface') \suchthat \genface' \subseteq \genface} \suchthat \genface \subseteq \stdface}\\
        & = \sum_{\genface \in \binom{\stdface}{j+1}}{\frac{1}{j+2}\sum_{\genface' \in \binom{\genface}{i+1}}{\frac{1}{\binom{j+1}{i+1}}\stdcochain(\genface')}}
        = \sum_{\genface \in \binom{\stdface}{j+1}}{\frac{1}{j+2}\sum_{\genface' \in \binom{\genface}{i+1}}{\frac{(j-i)!(i+1)!}{(j+1)!}\stdcochain(\genface')}}\\
        & = \sum_{\genface \in \binom{\stdface}{j+1}}{\sum_{\genface' \in \binom{\genface}{i+1}}{\frac{(j-i)!(i+1)!}{(j+2)!}\stdcochain(\genface')}}
        = \sum_{\genface \in \binom{\stdface}{i+1}}{(j-i+1)\frac{(j-i)!(i+1)!}{(j+2)!}\stdcochain(\genface)}\\
        & = \sum_{\genface \in \binom{\stdface}{i+1}}{\frac{(j-i+1)!(i+1)!}{(j+2)!}\stdcochain(\genface)}
        = \sum_{\genface \in \binom{\stdface}{i+1}}{\frac{1}{\binom{j+2}{i+1}}\stdcochain(\genface)}\\
        & = \ev{\genface \in \stdcomplex\parens{i}}{\stdcochain(\genface) \suchthat \genface \subseteq \stdface}
    \end{align*}
\end{proof}
We turn our attention to showing that the adjoint operator of the signless differential localizes to the links in the following way:
\begin{lemma}\label{lem:down-operator-and-localisation}
For every cochain $\stdcochain$ and every face $\genface$ it holds that:
\[
    \sldoperator_{\genface}^* \stdcochain_{\genface} = (\sldoperator^* \stdcochain)_{\genface}
\]
\end{lemma}
\begin{proof}
    The Lemma holds due to:
    \[
        \sldoperator_{\genface}^* \stdcochain_{\genface}(\stdface)
        = \ev{\genface' \in \parens{\stdcomplex_{\genface}}_{\stdface}(0)}{\parens{\stdcochain_{\genface}}_{\stdface}(\genface')}
        = \ev{\genface' \in \stdcomplex_{\genface\stdface}(0)}{\stdcochain_{\genface\stdface}(\genface')}
        = \sldoperator^* \stdcochain(\stdface\genface)
        = (\sldoperator^* \stdcochain)_{\genface} (\stdface)
    \]
\end{proof}
\begin{lemma}\label{lem:multi-sld-operator-adj}
Let $\stdcochain$ be an $i$-dimensional cochain.
Then:
\[
    \sldoperator^{*}_{i}\cdots \sldoperator^{*}_{j-1}\stdcochain(\stdface) = \ev{\genface \in \stdcomplex_{\stdface}(j-i-1)}{\stdcochain_{\stdface}(\genface)}
\]
\end{lemma}
\begin{proof}
    By induction, the case where $j=i+1$ holds due to Lemma~\ref{lem:sld-adj-explicit}.
    The induction step follows the following argument:
    \begin{align*}
        \sldoperator^{*}_{i}\cdots \sldoperator^{*}_{j}\stdcochain(\stdface)
        & = \ev{\stdvertex \in \stdcomplex_{\stdface}(0)}{\parens{\sldoperator^{*}_{i+1}\cdots \sldoperator^{*}_{j}\stdcochain}_{\stdface}(\stdvertex)}
        =\ev{\stdvertex \in \stdcomplex_{\stdface}(0)}{\sldoperator^{*}_{\stdface,i}\cdots \sldoperator^{*}_{\stdface,j-1}\stdcochain_{\stdface}(\stdvertex)}\\
        &=\ev{\stdvertex \in \stdcomplex_{\stdface}(0)}{\ev{\genface \in \stdcomplex_{\stdface\stdvertex}(j-i-1)}{\stdcochain_{\stdface\stdvertex}(\genface)}}
        = \sum_{\stdvertex \in \stdcomplex_{\stdface}(0)}{\weight[\stdface][\stdvertex]\sum_{\genface \in \stdcomplex_{\stdface\stdvertex}(j-i-1)}{\weight[\stdface\stdvertex][\genface]\stdcochain_{\stdface\stdvertex}(\genface)}}\\
        &= \sum_{\stdvertex \in \stdcomplex_{\stdface}(0)}{\sum_{\genface \in \stdcomplex_{\stdface\stdvertex}(j-i-1)}{\weight[\stdface][\stdvertex]\weight[\stdface\stdvertex][\genface]\stdcochain_{\stdface}(\genface\stdvertex)}}\\
        &= \sum_{\stdvertex \in \stdcomplex_{\stdface}(0)}{\sum_{\genface \in \stdcomplex_{\stdface\stdvertex}(j-i-1)}{\frac{\weight[\stdface][\genface\stdvertex]}{\binom{\abs{\genface\stdvertex}}{\abs{\stdvertex}}}\stdcochain_{\stdface}(\genface\stdvertex)}}\\
        &= \sum_{\genface' \in \stdcomplex_{\stdface}(j-i)}{\weight[\stdface][\genface']\stdcochain_{\stdface}(\genface')}\\
        &= \ev{\genface' \in \stdcomplex_{\stdface}(j-i)}{\stdcochain_{\stdface}(\genface')}
    \end{align*}
    The second equality is due to Lemma~\ref{lem:down-operator-and-localisation}.
\end{proof}
\section{Up-Down and Down-Up Operators}\label{sec:up-down-and-down-up-operators}
In this section we will present two objects of interest - the up-down and down-up walks.
These are natural operators that result from considering a standard walk on the $k$-dimensional faces of a simplicial complex using the following two steps:
A \emph{down step} where, given a $k$-dimensional face, the walk moves to a $(k-1)$-dimensional face that is contained in it with equal probability.
And an \emph{up step} in which, given a $k$-dimensional face, the walk moves to a $(k+1)$-dimensional face with probability proportional to its weight.
Applying the up step followed by the down step yields the up-down walk while applying the down step followed by the up step yields the down-up walk.
\begin{definition}[Up-down and down-up operators]
        Let $\stdcomplex$ be a simplicial complex.
    Then define the up-down random walk to be:
    \[
        \sparens{\upoperator}_{\stdface, \genface} = \begin{cases}
                                                         \frac{1}{k+2} & \stdface = \genface\\
                                                         \frac{\weight[\stdface][\genface \setminus \stdface]}{k+2} & \stdface \cup \genface \in \stdcomplex(k+1)\\
                                                         0 & \text{Otherwise}
        \end{cases}
    \]
    And the down-up random walk to be:
    \[
        \sparens{\downoperator}_{\stdface, \genface} = \begin{cases}
                                                           \frac{1}{k+1}\sum_{\genface' \in \binom{\stdface}{k}}{\weight[\genface'][\stdface \setminus \genface']} & \stdface = \genface\\
                                                           \frac{1}{k+1}\weight[\stdface \cap \genface][\genface \setminus \stdface] & \stdface \cap \genface \in \stdcomplex(k-1)\\
                                                           0 & \text{Otherwise}
        \end{cases}
    \]
\end{definition}
We would now like to present characterization the up-down random walk and the down-up random walk using the signless differential (this characterization had appeared in~\cite[Corollary~3.7]{DBLP:conf/approx/KaufmanO18} and is given here for completeness).
In order to do that, consider the following Lemma:
\begin{lemma}
    Let $\stdcomplex$ be a simplicial complex, it holds that:
    \[
        \upoperator[k]=\sldoperator^*_{k} \sldoperator_{k} \qquad \qquad \downoperator[k]=\sldoperator_{k-1} \sldoperator^{*}_{k-1}
    \]
\end{lemma}
\begin{proof}
    Let $\stdcochain \in \cochainset{k}{\stdcomplex;\mathbbm{R}}$ and $\stdface \in \stdcomplex(k)$, then:
    \begin{align*}
        \sldoperator^*\sldoperator \stdcochain(\stdface)
        & = \ev{\genface \in \stdcomplex_{\stdface}(0)}{\parens{\sldoperator\stdcochain}_{\stdface}(\genface)}
        = \sum_{\genface \in \stdcomplex_{\stdface}(0)}{\weight[\stdface][\genface]\parens{\sldoperator\stdcochain}_{\stdface}(\genface)}
        = \sum_{\genface \in \stdcomplex_{\stdface}(0)}{\weight[\stdface][\genface]\sldoperator\stdcochain(\genface \cup \stdface)} \\
        & = \sum_{\genface \in \stdcomplex_{\stdface}(0)}{\weight[\stdface][\genface]\frac{1}{k+2}\sum_{\substack{\genface' \in \stdcomplex(k) \\ \genface' \subseteq \stdface \cup \genface}}{\stdcochain(\genface')}} \\
        & = \frac{1}{k+2} \sum_{\genface \in \stdcomplex_{\stdface}(0)}{\sum_{\substack{\genface' \in \stdcomplex(k) \\ \genface' \subseteq \stdface \cup \genface}}{\weight[\stdface][\genface]\stdcochain(\genface')}} \\
        & = \frac{1}{k+2} \sum_{\genface \in \stdcomplex_{\stdface}(0)}{\parens{\weight[\stdface][\genface]\stdcochain(\stdface) + \sum_{\substack{\genface' \in \stdcomplex(k) \\ \genface' \subseteq \stdface \cup \genface \\ \genface' \ne \stdface}}{\weight[\stdface][\genface]\stdcochain(\genface')}}} \\
        & = \frac{1}{k+2} \parens{\sum_{\genface \in \stdcomplex_{\stdface}(0)}{\weight[\stdface][\genface]\stdcochain(\stdface)} + \sum_{\genface \in \stdcomplex_{\stdface}(0)}{\sum_{\substack{\genface' \in \stdcomplex(k) \\ \genface' \subseteq \stdface \cup \genface \\ \genface' \ne \stdface}}{\weight[\stdface][\genface]\stdcochain(\genface')}}} \\
        & = \frac{1}{k+2} \parens{\stdcochain(\stdface) + \sum_{\genface \in \stdcomplex_{\stdface}(0)}{\sum_{\substack{\genface' \in \stdcomplex(k) \\ \stdface \cup \genface' = \stdface \cup \genface}}{\weight[\stdface][\genface' \setminus \stdface]\stdcochain(\genface')}}} \\
        & = \frac{1}{k+2} \parens{\stdcochain(\stdface) + \sum_{\substack{\genface' \in \stdcomplex(k) \\ \stdface \cup \genface' \in \stdcomplex(k+1)}}{\weight[\stdface][\genface' \setminus \stdface]\stdcochain(\genface')}} \\
        & = \upoperator \stdcochain
    \end{align*}

    In addition:
    \begin{align*}
        \sldoperator\sldoperator^* \stdcochain(\stdface)
        & = \frac{1}{k+1} \sum_{\substack{\genface \in \stdcomplex(k-1) \\ \genface \subseteq \stdface}}{\sldoperator^* \stdcochain(\genface)}
        = \frac{1}{k+1} \sum_{\substack{\genface \in \stdcomplex(k-1) \\ \genface \subseteq \stdface}}{\ev{\genface' \in \stdcomplex_{\genface}(0)}{\stdcochain_{\genface}(\genface')}}\\
        & = \frac{1}{k+1} \sum_{\substack{\genface \in \stdcomplex(k-1) \\ \genface \subseteq \stdface}}{\sum_{\genface' \in \stdcomplex_{\genface}(0)}{\weight[\genface][\genface']\stdcochain_{\genface}(\genface')}}\\
        & = \frac{1}{k+1} \parens{\sum_{\substack{\genface \in \stdcomplex(k-1) \\ \genface \subseteq \stdface}}{\weight[\genface][\stdface]\stdcochain_{\genface}(\stdface \setminus \genface)} + \sum_{\substack{\genface \in \stdcomplex(k-1) \\ \genface \subseteq \stdface}}{\sum_{\substack{\genface' \in \stdcomplex_{\genface}(0) \\ \genface \cup \genface' \ne \stdface}}{\weight[\genface][\genface']\stdcochain_{\genface}(\genface')}}}\\
        & = \frac{1}{k+1} \parens{\sum_{\substack{\genface \in \stdcomplex(k-1) \\ \genface \subseteq \stdface}}{\weight[\genface][\stdface]\stdcochain(\stdface)} + \sum_{\substack{\genface \in \stdcomplex(k-1) \\ \genface \subseteq \stdface}}{\sum_{\substack{\genface'' \in \stdcomplex(k) \\ \genface'' \ne \stdface \\ \genface'' \cap \stdface = \genface}}{\weight[\genface][\genface'' \setminus \genface]\stdcochain_{\genface}(\genface'' \setminus \genface)}}}\\
        & = \frac{1}{k+1} \parens{\sum_{\substack{\genface \in \stdcomplex(k-1) \\ \genface \subseteq \stdface}}{\weight[\genface][\stdface]\stdcochain(\stdface)} + \sum_{\substack{\genface \in \stdcomplex(k-1) \\ \genface \subseteq \stdface}}{\sum_{\substack{\genface'' \in \stdcomplex(k) \\ \genface'' \cap \stdface = \genface}}{\weight[\genface][\genface'' \setminus \genface]\stdcochain(\genface'')}}}\\
        & = \frac{1}{k+1} \parens{\sum_{\substack{\genface \in \stdcomplex(k-1) \\ \genface \subseteq \stdface}}{\weight[\genface][\stdface]\stdcochain(\stdface)} + \sum_{\substack{\genface'' \in \stdcomplex(k) \\ \genface'' \cap \stdface \in \stdcomplex(k-1)}}{\weight[\genface'' \cap \stdface][\genface'' \setminus \stdface]\stdcochain(\genface'')}}\\
        & = \downoperator \stdcochain
    \end{align*}
\end{proof}

We are also going to be interested in applying the signless differential and its adjoint operator multiple times in a row.
We will therefore define the $k$-dimensional $i$-up-down operator and the $k$-dimensional $i$-down-up operator in the following way:
\begin{definition}
    Let $\upoperator[i][k]$ be the $k$-dimensional $i$-up-down operator and $\downoperator[i][k]$ be the $k$-dimensional $i$-down-up operator defined as follows:
    \[
        \upoperator[i][k]=\sldoperator^*_{k} \cdots \sldoperator^{*}_{k+i-1} \sldoperator_{k+i-1} \cdots \sldoperator_{k} \qquad \qquad \downoperator[i][k]=\sldoperator_{k-1} \cdots \sldoperator_{k-i-1} \sldoperator^{*}_{k-i-1} \cdots \sldoperator^{*}_{k-1}
    \]
\end{definition}
Recall that the $i$-up-down operator and the $i$-down-up operators can be presented explicitly using Lemma~\ref{lem:multi-sld-operator} and Lemma~\ref{lem:multi-sld-operator-adj}.
In addition, note that the $i$-up-down operator corresponds to applying the up step $i$ times and then applying the down step $i$ times.
Likewise the $i$-down-up operator corresponds to applying the down step $i$ times and then applying the up step $i$ times.

\section{The Non-Lazy Walk Operator}\label{sec:the-non-lazy-walk-operator}
Our main object of study is going to be the non-lazy $k$-dimensional random walk operator.
This operator is a generalization of the non-lazy random walk operator in graphs to higher dimensions.
In graphs the non-lazy random walk operator moves between two vertices if they have an edge connecting them.
The higher dimensional version of this operator is going to be something very similar:
It is going to move between two $k$-faces if there is a $(k+1)$-face that contains both faces.
\begin{definition}[The Non-Lazy $k$-dimensional Random Walk Operator]\label{def:non-lazy-up-down-random-walk}
    Let $\stdcomplex$ be a pure $d$-dimensional simplicial complex define the $k$-dimensional random walk operator to be the following operator:
    \[
        \sparens{\nlupoperator[k]}_{\stdface, \genface} = \begin{cases}
                                                              \frac{\weight[\stdface][\genface \setminus \stdface]}{k+1} & \stdface \cup \genface \in \stdcomplex(k+1)\\
                                                              0 & \text{Otherwise}
        \end{cases}
    \]
\end{definition}
One can also think of the non-lazy random walk operator as the regular up-down operator with the lazy part removed.
Formally:
\begin{observation}
    For every dimension $k$ it holds that:
    \[
        \nlupoperator[k] = \frac{k+2}{k+1} \upoperator - \frac{1}{k+1} I
    \]
\end{observation}
Of specific interest is the $0$-dimensional non-lazy up-down operator as it can be used to describe \emph{every} one of the $0$-dimensional $i$-up-down operators.
This is because every one of these operators ultimately move between a vertex and its neighbours.
The more steps one takes up the complex the more mixed the result is.
By that we mean that the non-lazy walk operator has a larger effect on the result.
Quantitatively, this can be formulated via the following Lemma.
\begin{lemma}\label{lem:nl-random-walk-and-i-up-operator}
    \[
        \nlupoperator[0] = \frac{i+1}{i} \upoperator[i][0] - \frac{1}{i}I
    \]
\end{lemma}
\begin{proof}
    Consider the up operator:
    \begin{align*}
        \upoperator[i][0] \stdcochain (\stdvertex)
        & = \sldoperator^*_{0} \cdots \sldoperator^*_{i-1}\sldoperator_{i-1} \cdots \sldoperator_{0} \stdcochain (\stdvertex)
        = \ev{\genface \in \stdcomplex_{\stdvertex}(i-1)}{\parens{\sldoperator_{i-1} \cdots \sldoperator_{0}\stdcochain}_{\stdvertex}(\genface)}=\\
        & = \ev{\genface \in \stdcomplex_{\stdvertex}(i-1)}{\sldoperator_{i-1} \cdots \sldoperator_{0}\stdcochain(\genface\stdvertex)}
        = \ev{\genface \in \stdcomplex_{\stdvertex}(i-1)}{\ev{\genvertex \in \stdcomplex\parens{0}}{\stdcochain(\genvertex) \suchthat \genvertex \subseteq \genface\stdvertex}}\\
        &= \sum_{\genface \in \stdcomplex_{\stdvertex}(i-1)}{\weight[\stdvertex][\genface]\sum_{\genvertex \in \genface\stdvertex}{\frac{1}{\binom{i+1}{1}}\stdcochain(\genvertex)}}
        = \sum_{\genface \in \stdcomplex_{\stdvertex}(i-1)}{\sum_{\genvertex \in \genface\stdvertex}{\frac{\weight[\stdvertex][\genface]}{i+1}\stdcochain(\genvertex)}}\\
        & = \sum_{\genface \in \stdcomplex_{\stdvertex}(i-1)}{\frac{\weight[\stdvertex][\genface]}{i+1}\stdcochain(\stdvertex)} + \sum_{\genface \in \stdcomplex_{\stdvertex}(i-1)}{\sum_{\genvertex \in \genface}{\frac{\weight[\stdvertex][\genface]}{i+1}\stdcochain(\genvertex)}}\\
        & = \frac{1}{i+1}\stdcochain(\stdvertex) + \sum_{\substack{\genface \in \stdcomplex(i) \\ \stdvertex \in \genface}}{\sum_{\substack{\genvertex \in \stdcomplex(0) \\ \genvertex \ne \stdvertex, \genvertex \in \genface}}{\frac{\weight[\genface]}{(i+1)^2\weight[\stdvertex]}\stdcochain(\genvertex)}}\\
        & = \frac{1}{i+1}\stdcochain(\stdvertex) + \sum_{\genvertex \in \stdcomplex(0)}{\sum_{\substack{\genface \in \stdcomplex(i) \\ \genvertex \ne \stdvertex \\ \genvertex,\stdvertex \in \genface}}{\frac{\weight[\genface]}{(i+1)^2 \weight[\stdvertex]}\stdcochain(\genvertex)}}\\
        & = \frac{1}{i+1}\stdcochain(\stdvertex) + \sum_{\substack{\genvertex \in \stdcomplex(0) \\ \stdvertex\genvertex \in \stdcomplex(1)}}{\sum_{\substack{\genface \in \stdcomplex(i) \\ \genvertex,\stdvertex \in \genface}}{\frac{\weight[\genface]}{(i+1)^2 \weight[\stdvertex]}\stdcochain(\genvertex)}}\\
        & = \frac{1}{i+1}\stdcochain(\stdvertex) + \sum_{\substack{\genvertex \in \stdcomplex(0) \\ \stdvertex\genvertex \in \stdcomplex(1)}}{\frac{1}{(i+1)^2 \weight[\stdvertex]}\stdcochain(\genvertex) \sum_{\substack{\genface \in \stdcomplex(i) \\ \genvertex,\stdvertex \in \genface}}{\weight[\genface]}}\\
        & = \frac{1}{i+1}\stdcochain(\stdvertex) + \sum_{\substack{\genvertex \in \stdcomplex(0) \\ \stdvertex\genvertex \in \stdcomplex(1)}}{\frac{\binom{i+1}{2}\weight[\stdvertex\genvertex]}{(i+1)^2 \weight[\stdvertex]}\stdcochain(\genvertex)}\\
        & = \frac{1}{i+1}\stdcochain(\stdvertex) + \sum_{\substack{\genvertex \in \stdcomplex(0) \\ \stdvertex\genvertex \in \stdcomplex(1)}}{\frac{i\weight[\stdvertex\genvertex]}{2 (i+1) \weight[\stdvertex]}\stdcochain(\genvertex)}\\
        & = \frac{1}{i+1}\stdcochain(\stdvertex) + \frac{i}{i+1}\sum_{\substack{\genvertex \in \stdcomplex(0) \\ \stdvertex\genvertex \in \stdcomplex(1)}}{\weight[\stdvertex][\genvertex]\stdcochain(\genvertex)}\\
        & = \frac{1}{i+1}\stdcochain(\stdvertex) + \frac{i}{i+1}\nlupoperator[0]\stdcochain(\stdvertex)
    \end{align*}
    Therefore:
    \[
        \nlupoperator[0] = \frac{i+1}{i} \upoperator[i][0] - \frac{1}{i}I
    \]
\end{proof}
\section{Analyzing the Non-Lazy Random Walk Operator}\label{sec:analyzing-the-non-lazy-random-walk-operator}
We are now ready to start analyzing the random walk operators.
In order to do so we are going to use a \emph{local to global} argument.
In order to apply a local to global argument we must understand how to view a cochain through the links.
Specifically we will be interested in viewing methods that satisfy the following:
\begin{definition}[Link viewer]
A \emph{link viewer} is any transformation $\viewer$ that accepts a face $\stdface$ and a cochain in $\stdcomplex$.
It then returns a cochain in $\stdcomplex_{\stdface}$.
In addition, a link viewer satisfies the following:
    \begin{itemize}
        \item For every face $\stdface$ it holds that $\viewer[\stdface]$ is linear.
        \item For every face $\stdface$ it holds that $\viewer[\stdface]\ind{} = \ind{}$.
        \item For every two cochains $\stdcochain, \gencochain$ and any two faces of the same dimension $\stdface$,$\genface$ it holds that:
        \[
            \dim\parens{\stdcochain} - \dim\parens{\viewer[\stdface]\stdcochain} = \dim\parens{\gencochain}-\dim\parens{\viewer[\genface]\gencochain}
        \]
        In addition, denote the dimensional difference for vertices by:
        \[
            \dimdiff = \dim\parens{\stdcochain} - \dim\parens{\viewer[\stdvertex]\stdcochain}
        \]
        \item Viewing of a cochain in a link is only determined by the cochain and the face for which it is the link (and not the path taken to achieve said view).
        Formally: For every $\genface \subseteq \stdface \in \stdcomplex$ it holds that:
        \[
            \viewer[\stdface] = \viewer[\stdface \setminus \genface]\viewer[\genface]
        \]
        \item For every dimension $i$ it holds that:
        \[
            \innerprod{\stdcochain, \gencochain} = \ev{\stdface \in \stdcomplex(i)}{\innerprod{\viewer[\stdface][\stdcochain], \viewer[\stdface][\gencochain]}}
        \]
    \end{itemize}
\end{definition}
Of particular interest are link viewers that view the non-lazy random walk operator in ``the right way'':
\begin{definition}
    A link viewer respects the non-lazy up-down random walk if for every $k$:
    \[
        \innerprod{\nlupoperator[k]\stdcochain, \stdcochain} = \ev{\stdvertex \in \stdcomplex(0)}{\innerprod{\nlupoperator[k-\dimdiff]\viewer[\stdvertex][\stdcochain], \viewer[\stdvertex][\stdcochain]}}
    \]
\end{definition}

\begin{definition}[$i$-level cochain]\label{def:i-level-cochain}
    A cochain $\stdcochain$ is an $i$-level cochain with respect to $\viewer$ if it holds that:
    \[
        \forall \stdface \in \stdcomplex(i-1): \innerprod{\viewer[\stdface][\stdcochain], \ind{}} = 0
    \]
    When the link viewer is clear from context we will simply refer to them as $i$-level cochains.
    Denote the set of $i$-level $j$-dimensional cochains in $\stdcomplex$ by $\levelcochainset{\viewer, i}{j}{\stdcomplex;\R}$.
\end{definition}
\begin{lemma}\label{lem:containment-of-i-level-cochains}
    For every dimension $j \le i$:
    \[
        \levelcochainset{\viewer, i}{k}{\stdcomplex;\R} \subseteq \levelcochainset{\viewer, j}{k}{\stdcomplex;\R}
    \]
\end{lemma}
\begin{proof}
    Let $\stdcochain \in \levelcochainset{\viewer, i}{k}{\stdcomplex;\R}$ and let $\stdface$ be a $(j-1)$-dimensional face and note the following:
    \begin{align*}
        \innerprod{\viewer[\stdface][\stdcochain], \ind{}}
        & = \ev{\genface \in \stdcomplex_{\stdface}(i-j)}{\innerprod{\viewer[\genface][\viewer[\stdface][\stdcochain]], \viewer[\genface][\ind{}]}}
        = \ev{\genface \in \stdcomplex_{\stdface}(i-j)}{\innerprod{\viewer[\genface\stdface][\stdcochain], \ind{}}}
        = 0
    \end{align*}
    Where the last equality is due to the fact that $\stdface \cup \genface$ is an $(i-1)$-dimensional face.
\end{proof}
Of particular interest are $i$-level cochains that are orthogonal to the $(i+1)$-level cochains.
We will therefore define the following:
\begin{definition}[Proper $i$-level cochain]
    A cochain $\stdcochain$ is a proper $i$-level cochain with respect to $\viewer$ if it holds that:
    \[
        \stdcochain \in \levelcochainset{\viewer, i}{j}{\stdcomplex;\R} \cap \parens{\levelcochainset{\viewer, i+1}{j}{\stdcomplex;\R}}^{\bot}
    \]
    When the link viewer is clear from context we will simply refer to them as proper $i$-level cochains.
    Denote the set of proper $i$-level $j$-dimensional cochains in $\stdcomplex$ by $\levelcochainset{\viewer, \hat{i}}{j}{\stdcomplex;\R}$.
\end{definition}
Consider the following key property of proper level cochains:
\begin{lemma}\label{lem:orthogonality-of-proper-level-cochains}
    Let $i < j$ and let $\stdcochain$ be a proper $i$-level cochain and $\gencochain$ be a proper $j$-level cochain then:
    \[
        \innerprod{\stdcochain, \gencochain} = 0
    \]
\end{lemma}
\begin{proof}
    Due to Lemma~\ref{lem:containment-of-i-level-cochains} it holds that $\gencochain$ is also a $i+1$ level cochain and thus is orthogonal to $\gencochain$ by definition.
\end{proof}
Note that considering proper level cochains is one of the key ingredients of this paper.
In previous results (For example~\cite{DBLP:conf/approx/KaufmanO18}) instead of considering proper $i$-level cochains the authors essentially considered cochains in $\parens{\levelcochainset{\viewer,i}{k}{\stdcomplex;\R}}^{\bot}$.
Using proper level cochains allows us to separate the different levels completely and achieve a proper decomposition.
Since analyzing the decomposition with only pure level cochains in mind is hard (as even if $\stdcochain$ is a level cochain the same cannot be said about $\viewer[\stdface][\stdcochain]$) we resort to first show the following technical theorem.
\begin{theorem}[Bootstrapping Theorem, Technical Version]\label{thm:main-technichal}
    Let $\stdcomplex$ be a pure $d$-dimensional simplicial complex, $\viewer$ be a link viewer that respects the non-lazy up-down random walk and for every $i$ let $\stdcochain_i \in \levelcochainset{\viewer, i}{k}{\stdcomplex;\R}$ and $\stdcochain = \sum_{i=0}^p{\stdcochain_i}$.
    Denote $r = k-\dimdiff$ and suppose that there are values of $\set{\lambda_{\stdface,i,j}}_{\stdface \in \stdcomplex, i \in \sparens{d}, j \in \sparens{d}}$ ($\lambda_{\stdface,i,j}$ is the contraction of an $i$-level $j$-dimensional cochain in the link of $\stdface$) such that for every $\stdface \in \stdcomplex$ and $\gencochain_0$ a $0$-level cochain:
    \[
    \begin{cases}
        \lambda_{\stdface, 1, k}\norm{\gencochain_0}^2 + \ev{\stdvertex \in \stdcomplex_{\stdface}(0)}{(1-\lambda_{\stdface\cup\set{\stdvertex}, 1, r})\norm{\downoperator[r][r]\viewer[\stdvertex][\gencochain_0]}^2} \le \lambda_{\stdface, 0, k} \norm{\gencochain_0}^2\\
         \max_{\stdvertex \in \stdcomplex_{\stdface}(0)}{\set{\lambda_{\stdface \cup \set{\stdvertex}, i-1, r}}} \le \lambda_{\stdface, i, k}
    \end{cases}
    \]
    Then:
    \[
        \innerprod{\nlupoperator[k]\stdcochain, \stdcochain} \le \sum_{i=0}^p{\lambda_{\emptyset, i, k}\norm{\stdcochain_i}^2} +\sum_{i = 0}^p{\sum_{\substack{j=1\\i < j}}^p{c_{i,j}\innerprod{\stdcochain_i, \stdcochain_j}}}
    \]

    For some constants $\set{c_{i,j}}$ and where $p$ is the number of level functions that span the space orthogonal to the constants.
\end{theorem}
\begin{proof}
    Let $r$ be the dimension of $\viewer[\stdface][\stdcochain]$.
    Consider the following:
    \begin{equation}\label{eq:main-thm-expected-value}
    \begin{aligned}
        \innerprod{\nlupoperator[k]\stdcochain, \stdcochain}
        & = \innerprod{\nlupoperator[k]\sum_{i=0}^p{\stdcochain_i}, \sum_{i=0}^p{\stdcochain_i}}
        = \ev{\stdvertex \in \stdcomplex(0)}{\innerprod{\viewer[\stdvertex][\parens{\nlupoperator[r]\sum_{i=0}^p{\stdcochain_i}}], \viewer[\stdvertex][\parens{\sum_{i=0}^p{\stdcochain_i}}]}}\\
        & = \ev{\stdvertex \in \stdcomplex(0)}{\innerprod{\nlupoperator[r]\sum_{i=0}^p{\viewer[\stdvertex][\stdcochain_i]}, \sum_{i=0}^p{\viewer[\stdvertex][\stdcochain_i]}}}\\
        & = \ev{\stdvertex \in \stdcomplex(0)}{\innerprod{\nlupoperator[r]\parens{I-\downoperator[r][r]}\sum_{i=0}^p{\viewer[\stdvertex][\stdcochain_i]}, \parens{I-\downoperator[r][r]}\sum_{i=0}^p{\viewer[\stdvertex][\stdcochain_i]}}} + \\
        & \qquad\qquad + \ev{\stdvertex \in \stdcomplex(0)}{\innerprod{\nlupoperator[r]\downoperator[r][r]\sum_{i=0}^p{\viewer[\stdvertex][\stdcochain_i]}, \downoperator[r][r]\sum_{i=0}^p{\viewer[\stdvertex][\stdcochain_i]}}}\\
        & = \ev{\stdvertex \in \stdcomplex(0)}{\innerprod{\nlupoperator[r]\parens{I-\downoperator[r][r]}\sum_{i=0}^p{\viewer[\stdvertex][\stdcochain_i]}, \parens{I-\downoperator[r][r]}\sum_{i=0}^p{\viewer[\stdvertex][\stdcochain_i]}}} + \\ &\qquad \qquad + \ev{\stdvertex \in \stdcomplex(0)}{\norm{\downoperator[r][r]\viewer[\stdvertex][\stdcochain_0]}^2}\\
    \end{aligned}
    \end{equation}
    We will now like to apply our Theorem to the localized cochains in the links.
    For that, consider the following level cochains:
    \begin{center}
    \begin{tabular}{c|c|c}
         level & cochain & square of norm \\
         \hline
         $k-1$   & $\stdcochain^{\stdvertex}_{k-1}:=\viewer[\stdvertex][\stdcochain_k]$ & $\norm{\viewer[\stdvertex][\stdcochain_k]}^2$\\
         $\vdots$ & $\vdots$ & $\vdots$\\
         $1$ & $\stdcochain^{\stdvertex}_{1}:=\viewer[\stdvertex][\stdcochain_2]$ & $\norm{\viewer[\stdvertex][\stdcochain_2]}^2$\\
         $0$ & $\stdcochain^{\stdvertex}_{0}:=\viewer[\stdvertex][\stdcochain_1] + \parens{I-\downoperator[r][r]}\viewer[\stdvertex][\stdcochain_0]$ & $\norm{\viewer[\stdvertex][\stdcochain_1]}^2 + \norm{\parens{I-\downoperator[r][r]}\viewer[\stdvertex][\stdcochain_0]}^2 + 2 \innerprod{\viewer[\stdvertex][\stdcochain_1], \viewer[\stdvertex][\stdcochain_0]}$
    \end{tabular}
    \end{center}
    Note that, by definition, for every $i \ge 2$ it holds that $\viewer[\stdvertex][\stdcochain_i] \in \levelcochainset{\viewer,i-1}{r}{\stdcomplex_{\stdvertex};\R}$.
    In addition $\viewer[\stdvertex][\stdcochain_1]+\parens{I-\downoperator[r][r]}\viewer[\stdvertex][\stdcochain_0] \in \levelcochainset{\viewer, 0}{r}{\stdcomplex_{\stdvertex};\R}$.
    We can therefore apply the Theorem to every link which (after some manipulation of the mixed terms that, for completion, is presented in Lemma~\ref{lem:dealing-with-mixed-terms}) yields that, for every link $\stdvertex$, it holds that:
    \begin{multline*}
        \innerprod{\nlupoperator[\stdvertex, k]\parens{I-\downoperator[r][r]}\viewer[\stdvertex][\stdcochain], \parens{I-\downoperator[r][r]}\viewer[\stdvertex][\stdcochain]} 
        \le \sum_{i=0}^{k-1}{\lambda_{\stdvertex, i, r}\norm{\stdcochain^{\stdvertex}_i}}^2 + \sum_{i = 0}^p{\sum_{\substack{j=1\\i < j}}^p{c_{i,j}\innerprod{\stdcochain^{\stdvertex}_i, \stdcochain^{\stdvertex}_j}}}=\\
        = \sum_{i=0}^{k-1}{\lambda_{\stdvertex, i, r}\norm{\viewer[\stdvertex][\stdcochain_i]}}^2 + \lambda_{\stdvertex, 0, r}\norm{\parens{I-\downoperator[r][r]}\viewer[\stdvertex][\stdcochain_0]}^2 + \sum_{i = 0}^p{\sum_{\substack{j=1\\i < j}}^p{c'_{i,j}\innerprod{\viewer[\stdvertex][\stdcochain_i], \viewer[\stdvertex][\stdcochain_j]}}}
    \end{multline*}
    Combining this with~\ref{eq:main-thm-expected-value} and noting that $\lambda_{\emptyset, i, k} \ge \max_{\stdvertex \in \stdcomplex(0)}{\set{\lambda_{\stdvertex, i-1, r}}}$ yields that:
    \begin{multline*}
        \innerprod{\nlupoperator[k]\stdcochain, \stdcochain} \le
        \sum_{i=1}^{k}{\lambda_{\emptyset, i, k}\norm{\stdcochain_i}}^2 + \ev{\stdvertex \in \stdcomplex(0)}{\norm{\downoperator[r][r]\viewer[\stdvertex][\stdcochain_0]}^2 + \lambda_{\stdvertex, 0, r}\norm{\parens{I-\downoperator[r][r]}\viewer[\stdvertex][\stdcochain_0]}^2} + \\ + \sum_{i = 0}^p{\sum_{\substack{j=1\\i < j}}^p{c'_{i,j}\innerprod{\stdcochain_i, \stdcochain_j}}}
    \end{multline*}
    Therefore all we have to do is prove that:
    \[
        \ev{\stdvertex \in \stdcomplex(0)}{\norm{\downoperator[r][r]\viewer[\stdvertex][\stdcochain_0]}^2 + \lambda_{\stdvertex, 0, r}\norm{\parens{I-\downoperator[r][r]}\viewer[\stdvertex][\stdcochain_0]}^2} \le
        \lambda_{\emptyset, 0, k} \norm{\stdcochain_0}^2
    \]
    This follows directly from our choice of $\lambda$s:
    \begin{multline*}
        \ev{\stdvertex \in \stdcomplex(0)}{\norm{\downoperator[r][r]\viewer[\stdvertex][\stdcochain_0]}^2 + \lambda_{\stdvertex, 0, r}\norm{\parens{I-\downoperator[r][r]}\viewer[\stdvertex][\stdcochain_0]}^2}=\\
        = \ev{\stdvertex \in \stdcomplex(0)}{\parens{1-\lambda_{\stdvertex, 0, r}}\norm{\downoperator[r][r]\viewer[\stdvertex][\stdcochain_0]}^2 + \lambda_{\stdvertex, 0, r}\parens{\norm{\downoperator[r][r]\viewer[\stdvertex][\stdcochain_0]}^2 + \norm{\parens{I-\downoperator[r][r]}\viewer[\stdvertex][\stdcochain_0]}^2}}=\\
        = \ev{\stdvertex \in \stdcomplex(0)}{\parens{1-\lambda_{\stdvertex, 0, r}}\norm{\downoperator[r][r]\viewer[\stdvertex][\stdcochain_0]}^2 + \lambda_{\stdvertex, 0, r}\norm{\viewer[\stdvertex][\stdcochain_0]}^2} \le \\
        \le \ev{\stdvertex \in \stdcomplex(0)}{\parens{1-\lambda_{\stdvertex, 0, r}}\norm{\downoperator[r][r]\viewer[\stdvertex][\stdcochain_0]}^2} + \lambda_{\emptyset, 1, k}\norm{\stdcochain_0}^2 \le \lambda_{\emptyset, 0, k} \norm{\stdcochain_0}^2
    \end{multline*}
\end{proof}
\begin{theorem}[Bootstrapping Theorem]\label{thm:main}
    Let $\stdcomplex$ be a pure $d$-dimensional simplicial complex, $\viewer$ be a link viewer that respects the non-lazy up-down random walk and for every $i$ let $\stdcochain_i \in \levelcochainset{\viewer,\hat{i}}{k}{\stdcomplex;\R}$ be a proper $i$-level cochain and $\stdcochain = \sum_{i=0}^p{\stdcochain_i}$.
    Denote $r = k-\dimdiff$ and suppose that there are values of $\set{\lambda_{\stdface,i,j}}_{\stdface \in \stdcomplex, i \in \sparens{d}, j \in \sparens{d}}$ such that for every $\stdface \in \stdcomplex$ and $\gencochain_0$ a $0$-level cochain:
    \[
    \begin{cases}
        \lambda_{\stdface, 1, k}\norm{\gencochain_0}^2 + \ev{\stdvertex \in \stdcomplex_{\stdface}(0)}{(1-\lambda_{\stdface\cup\set{\stdvertex}, 1, r})\norm{\downoperator[r][r]\viewer[\stdvertex][\gencochain_0]}^2} \le \lambda_{\stdface, 0, k} \norm{\gencochain_0}^2\\
         \max_{\stdvertex \in \stdcomplex_{\stdface}(0)}{\set{\lambda_{\stdface \cup \set{\stdvertex}, i-1, r}}} \le \lambda_{\stdface, i, k}
    \end{cases}
    \]
    Then:
    \[
        \innerprod{\nlupoperator[k]\stdcochain, \stdcochain} \le \sum_{i=0}^p{\lambda_{\emptyset, i, k}\norm{\stdcochain_i}^2}
    \]
\end{theorem}
\begin{proof}
    Note that the difference between this Theorem and Theorem~\ref{thm:main-technichal} is the choice of $\stdcochain_i$s.
    Specifically, in this Theorem the the cochains $\stdcochain_i$ are chosen to be \emph{proper} $i$-level cochains.
    Therefore, due to Lemma~\ref{lem:orthogonality-of-proper-level-cochains}, they are orthogonal to each other.
    This allows us to apply Theorem~\ref{thm:main-technichal} and note that:
    \[
        \innerprod{\nlupoperator[k]\stdcochain, \stdcochain} \le \sum_{i=0}^p{\lambda_{\emptyset, i, k}\norm{\stdcochain_i}^2} +\sum_{i = 0}^p{\sum_{\substack{j=1\\i < j}}^p{{c''}_{i,j}\innerprod{\stdcochain_i, \stdcochain_j}}} = \sum_{i=0}^p{\lambda_{\emptyset, i, k}\norm{\stdcochain_i}^2}
    \]
\end{proof}
It is important to note that, unlike the decomposition known in the two-sided case, this decomposition is not a decomposition to approximate eigenfunctions.
When applying the walk operator to a level function the result might be spread over multiple levels.
Theorem~\ref{thm:main} also yields a decomposition to the up-down operator:
\begin{corollary}
    With the same assumptions as Theorem~\ref{thm:main} it holds that
    \[
        \innerprod{\upoperator[k]\stdcochain, \stdcochain} \le \sum_{i=0}^p{\parens{\frac{k+1}{k+2}\lambda_{\emptyset, i, k} - \frac{1}{k+1}}\norm{\stdcochain_i}^2}
    \]
\end{corollary}
\begin{proof}
    The following holds:
    \begin{multline*}
        \innerprod{\upoperator[k]\stdcochain, \stdcochain}
        = \sum_{i=0}^p{\frac{k+2}{k+1} \innerprod{\nlupoperator[k]\stdcochain, \stdcochain} - \frac{1}{k+1} \innerprod{\stdcochain, \stdcochain}}
        \le \sum_{i=0}^p{\frac{k+2}{k+1} \innerprod{\lambda_{\emptyset,i,k}\stdcochain, \stdcochain} - \frac{1}{k+1} \norm{\stdcochain}^2}=\\
        = \sum_{i=0}^p{\frac{k+2}{k+1} \lambda_{\emptyset,i,k}\norm{\stdcochain}^2 - \frac{1}{k+1} \norm{\stdcochain}^2}
        = \sum_{i=0}^p{\parens{\frac{k+1}{k+2}\lambda_{\emptyset, i, k} - \frac{1}{k+1}}\norm{\stdcochain_i}^2}
    \end{multline*}
\end{proof}
Before we finish this section, we prove the following (fairly technical) Lemma for completion
\begin{lemma}\label{lem:dealing-with-mixed-terms}
    With the notations in Theorem~\ref{thm:main-technichal} and $L=\set{\stdcochain^{\stdvertex}_i \suchthat 1 \le i \le p}$ it holds that for every set of constants $\set{c_{\gencochain, \gencochain'}}_{\substack{\gencochain, \gencochain' \in L \\ \gencochain \ne \gencochain'}}$ there exists constants $\set{c'_{\gencochain, \gencochain'}}_{\substack{\gencochain, \gencochain' \in L \\ \gencochain \ne \gencochain'}}$ such that:
    \[
        \sum_{\substack{\gencochain, \gencochain' \in L\\\gencochain \ne \gencochain'}}{c_{\gencochain, \gencochain'}\innerprod{\gencochain, \gencochain'}}
        = \sum_{i = 0}^p{\sum_{\substack{j=1\\i < j}}^p{{c'}_{i,j}\innerprod{\viewer[\stdvertex][\stdcochain_i], \viewer[\stdvertex][\stdcochain_j]}}}
    \]
\end{lemma}
\begin{proof}
Consider the following:
        \begin{multline*}
        \sum_{\substack{\gencochain, \gencochain' \in L\\\gencochain \ne \gencochain'}}{c_{\gencochain, \gencochain'}\innerprod{\gencochain, \gencochain'}}
        = \sum_{i = 2}^p{\sum_{\substack{j=2\\i < j}}^p{c_{i,j}\innerprod{\viewer[\stdvertex][\stdcochain_i], \viewer[\stdvertex][\stdcochain_j]}}} + \sum_{i=2}^p{c_{0,i}\innerprod{\viewer[\stdvertex][\stdcochain_i], \viewer[\stdvertex][\stdcochain_1]+\parens{I-\downoperator[r][r]}\viewer[\stdvertex][\stdcochain_0]}}\\
        = \sum_{i = 2}^p{\sum_{\substack{j=2\\i < j}}^p{c_{i,j}\innerprod{\viewer[\stdvertex][\stdcochain_i], \viewer[\stdvertex][\stdcochain_j]}}} + \sum_{i=2}^p{c_{0,i}\innerprod{\viewer[\stdvertex][\stdcochain_i], \viewer[\stdvertex][\stdcochain_1]+\viewer[\stdvertex][\stdcochain_0]}} - \sum_{i=2}^p{c_{0,i}\innerprod{\viewer[\stdvertex][\stdcochain_i], \downoperator[r][r]\viewer[\stdvertex][\stdcochain_0]}}\\
        = \sum_{i = 2}^p{\sum_{\substack{j=2\\i < j}}^p{c_{i,j}\innerprod{\viewer[\stdvertex][\stdcochain_i], \viewer[\stdvertex][\stdcochain_j]}}} + \sum_{i=2}^p{c_{0,i}\innerprod{\viewer[\stdvertex][\stdcochain_i], \viewer[\stdvertex][\stdcochain_1]}} + \sum_{i=2}^p{c_{0,i}\innerprod{\viewer[\stdvertex][\stdcochain_i], \viewer[\stdvertex][\stdcochain_0]]}}\\
        = \sum_{i = 0}^p{\sum_{\substack{j=1\\i < j}}^p{{c'}_{i,j}\innerprod{\viewer[\stdvertex][\stdcochain_i], \viewer[\stdvertex][\stdcochain_j]}}}
    \end{multline*}
\end{proof}
\section{Trickling Down}\label{sec:trickling-down}
Before we present our random walk decomposition Theorem, let us start with a ``warm up'':
An alternative proof for the trickling down theorem~\cite{https://doi.org/10.48550/arxiv.1709.04431} that is based on Theorem~\ref{thm:main}.
We believe that the fact that the main tool presented here can be used to prove the trickling down theorem is of independent interest: it shows that there is a single, local to global argument at the heart of both claims.
In the trickling down theorem we are interested in the connection between the $0$-dimensional non-lazy random walk on the vertices of a complex and the $0$-dimensional non-lazy up-down random walk on the links of the vertices of the complex.
It will, therefore, be natural to consider a link viewer that does not incur a decrease in dimension.
One such link viewer is the restriction link viewer defined as:
\begin{definition}[Restriction]
    Given a simplicial complex $\stdcomplex$ and a cochain $\stdcochain \in \cochainset{k}{\stdcomplex; \R}$ define the restriction link viewer in the following way:
    \[
        \forall \stdface \in \stdcomplex: \resviewer[\stdface]\stdcochain(\genface) = \stdcochain(\genface)
    \]
\end{definition}
This link viewer maps the $0$-dimensional non-lazy up-down random walk to the $0$-dimensional non-lazy up-down on the links of the vertices.
We will show that applying Theorem~\ref{thm:main} to the restriction link viewer yields the trickling down theorem.
However, before applying Theorem~\ref{thm:main}, we must first show that restriction is indeed a link viewer that respects the non-lazy up-down random walk.
\begin{lemma}
    The restriction link viewer is a link viewer.
\end{lemma}
\begin{proof}
    We will prove the properties point by point:
    \begin{itemize}
        \item $\resviewer[\stdface][\parens{\stdcochain + c \cdot \gencochain}](\genface) = \parens{\stdcochain + c \cdot \gencochain}(\genface) = \stdcochain(\genface) + c \cdot \gencochain(\genface) = \resviewer[\stdface][\stdcochain](\genface) + c \cdot \resviewer[\stdface][\gencochain](\genface)$
        \item $\resviewer[\stdface][\ind{}] = \ind{}$
        \item It holds that for every cochain $\stdcochain$ and every face $\stdface$:
        \[
            \dim\parens{\stdcochain} - \dim\parens{\resviewer[\stdface][\stdcochain]} = \dim\parens{\stdcochain} - \dim\parens{\stdcochain} = 0
        \]
        \item Let $\genface \subseteq \stdface \in \stdcomplex$ then:
        \[
            \resviewer[\stdface\setminus\genface][\resviewer[\genface][\stdcochain]](\genface')
            = \resviewer[\genface][\stdcochain](\genface')
            = \stdcochain(\genface')
            = \resviewer[\stdface][\stdcochain](\genface')
        \]
        \item For every dimension $i$ and $j=\dim\parens{\viewer[\stdface][\stdcochain]}$ it holds that:
        \begin{multline*}
            \ev{\stdface \in \stdcomplex(i)}{\innerprod{\resviewer[\stdface][\stdcochain], \resviewer[\stdface][\gencochain]}}
            = \sum_{\stdface \in \stdcomplex(i)}{\weight[\stdface]\innerprod{\resviewer[\stdface][\stdcochain], \resviewer[\stdface][\gencochain]}} = \\
            = \sum_{\stdface \in \stdcomplex(i)}{\weight[\stdface]\sum_{\genface \in \stdcomplex_{\stdface}(j)}{\weight[\stdface][\genface]\resviewer[\stdface][\stdcochain](\genface) \resviewer[\stdface][\gencochain](\genface)}} =\\
            = \sum_{\stdface \in \stdcomplex(i)}{\weight[\stdface]\sum_{\substack{\genface \in \stdcomplex(j)\\ \stdface \cup \genface \in \stdcomplex(i+j+1)}}{\weight[\stdface][\genface]\resviewer[\stdface][\stdcochain](\genface) \resviewer[\stdface][\gencochain](\genface)}} =\\
            = \sum_{\stdface \in \stdcomplex(i)}{\weight[\stdface]\sum_{\substack{\genface \in \stdcomplex(j)\\ \stdface \cup \genface \in \stdcomplex(i+j+1)}}{\weight[\stdface][\genface]\stdcochain(\genface) \gencochain(\genface)}} = \\
            = \sum_{\genface \in \stdcomplex(j)}{\weight[\stdface]\sum_{\substack{\stdface \in \stdcomplex(i)\\ \stdface \cup \genface \in \stdcomplex(i+j+1)}}{\weight[\stdface][\genface]\stdcochain(\genface) \gencochain(\genface)}} =\\
            = \sum_{\genface \in \stdcomplex(j)}{\stdcochain(\genface) \gencochain(\genface) \sum_{\substack{\stdface \in \stdcomplex(i)\\ \stdface \cup \genface \in \stdcomplex(i+j+1)}}{\weight[\stdface]\weight[\stdface][\genface]}} =\\
            = \sum_{\genface \in \stdcomplex(j)}{\weight[\genface] \stdcochain(\genface) \gencochain(\genface)}
            = \innerprod{\stdcochain, \gencochain}
        \end{multline*}
    \end{itemize}
\end{proof}

\begin{lemma}
    The restriction link viewer respects the non-lazy random walk operator.
\end{lemma}
\begin{proof}
    Note that, for every $i$-dimensional cochian $\gencochain$ it holds that:
    \[
        \sldoperator \resviewer[\stdface][\gencochain](\genface)
        = \sum_{\genface' \in \binom{\genface}{i}}{\frac{1}{i+1}\resviewer[\stdface][\gencochain](\genface')}
        = \sum_{\genface' \in \binom{\genface}{i}}{\frac{1}{i+1}\gencochain(\genface')}
        = \sldoperator \gencochain(\genface)
        = \resviewer[\stdface][\parens{\sldoperator\gencochain}](\genface)
    \]
    It therefore holds, for every dimension $k$:
    \begin{multline*}
        \innerprod{\upoperator[k]\stdcochain, \gencochain}
        = \innerprod{\sldoperator^*_k\sldoperator_k\stdcochain, \gencochain}
        = \innerprod{\sldoperator_k\stdcochain, \sldoperator_k\gencochain}
        = \ev{\stdface \in \stdcomplex(i)}{\innerprod{\resviewer[\stdface][\sldoperator_k\stdcochain], \resviewer[\stdface][\sldoperator_k\gencochain]}}\\
        = \ev{\stdface \in \stdcomplex(i)}{\innerprod{\sldoperator_k\resviewer[\stdface][\stdcochain], \sldoperator_k\resviewer[\stdface][\gencochain]}}
        = \ev{\stdface \in \stdcomplex(i)}{\innerprod{\sldoperator^*_k\sldoperator_k\resviewer[\stdface][\stdcochain], \resviewer[\stdface][\gencochain]}}
        = \ev{\stdface \in \stdcomplex(i)}{\innerprod{\upoperator[k]\resviewer[\stdface][\stdcochain], \resviewer[\stdface][\gencochain]}}
    \end{multline*}
    And thus:
    \begin{multline*}
        \innerprod{\nlupoperator[k]\stdcochain, \gencochain}
        = \frac{k+2}{k+1} \innerprod{\upoperator[k]\stdcochain, \gencochain} - \frac{1}{k+1} \innerprod{\stdcochain, \gencochain}\\
        = \frac{k+2}{k+1} \ev{\stdface \in \stdcomplex(i)}{\innerprod{\upoperator[k]\resviewer[\stdface][\stdcochain], \resviewer[\stdface][\gencochain]}} - \frac{1}{k+1} \ev{\stdface \in \stdcomplex(i)}{\innerprod{\resviewer[\stdface][\stdcochain], \resviewer[\stdface][\gencochain]}}\\
        = \ev{\stdface \in \stdcomplex(i)}{\frac{k+2}{k+1}\innerprod{\upoperator[k]\resviewer[\stdface][\stdcochain], \resviewer[\stdface][\gencochain]}- \frac{1}{k+1}\innerprod{\resviewer[\stdface][\stdcochain], \resviewer[\stdface][\gencochain]}}
        = \ev{\stdface \in \stdcomplex(i)}{\innerprod{\nlupoperator[k]\resviewer[\stdface][\stdcochain], \resviewer[\stdface][\gencochain]}}
    \end{multline*}
\end{proof}
\begin{lemma}\label{lem:trickling-down-advantage-technical}
    It holds for every vertex $\stdvertex$ and every cochain $\stdcochain \in \cochainset{0}{\stdcomplex;\R}$ that:
    \[
        \downoperator[\stdvertex, 0]\resviewer[\stdvertex][\stdcochain_0]
        = \innerprod{\resviewer[\stdvertex][\stdcochain], \ind{}}_{\stdvertex}
        = \ev{\genvertex \in \stdcomplex_{\stdvertex}(0)}{\resviewer[\stdvertex][\stdcochain](\genvertex)} 
        = \nlupoperator[0]\stdcochain(\stdvertex) 
    \]
\end{lemma}
\begin{proof}
    \begin{multline*}
        \nlupoperator[0]\stdcochain(\stdvertex)
        = \sparens{\nlupoperator[0]\stdcochain}_{\stdvertex}
        = \sum_{\genvertex \in \stdcomplex(0)}{\sparens{\nlupoperator[0]}_{\stdvertex, \genvertex}\stdcochain(\genvertex)}= \\
        = \sum_{\substack{\genvertex \in \stdcomplex(0)\\\genvertex \cup \stdvertex \in \stdcomplex(1)}}{\weight[\stdvertex][\genvertex \setminus \stdvertex]\stdcochain(\genvertex)} 
         = \sum_{\genvertex \in \stdcomplex_{\stdvertex}(0)}{\weight[\stdvertex][\genvertex]\resviewer[\stdvertex][\stdcochain](\genvertex)}
        = \ev{\genvertex \in \stdcomplex_{\stdvertex}(0)}{\resviewer[\stdvertex][\stdcochain](\genvertex)}
    \end{multline*}
\end{proof}
\begin{corollary}[The advantage]\label{cor:trickling-down-advantage}
    If $\stdcomplex$'s $1$-skeleton is a $\lambda$-spectral expander it holds for every vertex $\stdvertex$ and every cochain $\stdcochain \in \cochainset{0}{\stdcomplex;\R}$ that:
    \[
        \norm{\sldoperator^*_0\resviewer[\stdvertex][\stdcochain_0]} 
        = \norm{\downoperator[0]\resviewer[\stdvertex][\stdcochain_0]} 
        \le \lambda^2 \norm{\stdcochain}^2
    \]
\end{corollary}
\begin{proof}
    Note that the $1$-skeleton of $\stdcomplex$ is a $\lambda$-spectral expander and $\stdcochain \in \cochainset{0}{\stdcomplex;\R}$ therefore $\norm{\nlupoperator[0]\stdcochain}^2 \le \lambda^2 \norm{\stdcochain}^2$.
    Combining this with Lemma~\ref{lem:trickling-down-advantage-technical} yields:
    \[
        \norm{\downoperator[0]\resviewer[\stdvertex][\stdcochain_0]}^2
        = \norm{\nlupoperator[0]\stdcochain}^2
        = \norm{\nlupoperator[0]\stdcochain_0}^2
        \le \lambda^2 \norm{\stdcochain_0}^2
    \]
\end{proof}
We can now show how applying Theorem~\ref{thm:main} to the restriction link viewer yields Oppenheim's trickling down theorem~\cite[Theorem~4.1]{https://doi.org/10.48550/arxiv.1709.04431}.
\begin{theorem}[Trickling Down, restated Theorem~\ref{thm:trickling-down}]
    If it holds that:
    \begin{itemize}
        \item For every vertex $\stdvertex$: $\stdcomplex_{\stdvertex}$ is a $\lambda_{\stdvertex,0,k}$ spectral expander.
        \item $\stdcomplex$ is connected.
    \end{itemize}
    Then it holds that $\lambda_{\emptyset, 0, k} = \frac{\lambda_{\emptyset, 1, k}}{1-\lambda_{\emptyset, 1, k}}$.
\end{theorem}
\begin{proof}
    Consider the following:
    \[
        \lambda_{\emptyset, 1, k}\norm{\stdcochain_0}^2 + (1-\lambda_{\emptyset, 1, k})\ev{\stdvertex \in \stdcomplex(0)}{\norm{\downoperator[\stdvertex, 0]\resviewer[\stdvertex][\stdcochain_0]}^2} \le \lambda_{\emptyset, 0, k} \norm{\stdcochain_0}^2
    \]
    Using Corollary~\ref{cor:trickling-down-advantage} it suffices to find values of $\lambda_{\stdface, i,j}$ such that:
    \[
        \lambda_{\emptyset, 1, k}\norm{\stdcochain_0}^2 + (1-\lambda_{\emptyset, 1, k})\lambda_{\emptyset, 0, k}^2\norm{\stdcochain_0}^2 \le \lambda_{\emptyset, 0, k} \norm{\stdcochain_0}^2
    \]
    Therefore solving the following inequality would bound $\lambda_{\emptyset, 0, k}$:
    \[
        \lambda_{\emptyset, 1, k} + (1-\lambda_{\emptyset, 1, k})\lambda_{\emptyset, 0, k}^2 \le \lambda_{\emptyset, 0, k}
    \]
    Note that picking $\lambda_{\emptyset, 0, k} = \frac{\lambda_{\emptyset, 1, k}}{1-\lambda_{\emptyset, 1, k}}$ satisfies the inequality and thus proves the Theorem.
\end{proof}
Before we end this section we would like to give a few words on why this proof does not yield a decomposition of the $0$-dimensioanl random walk operator that improves upon the Theorem's original formulation.
Consider how the $1$-level cochains behave when applying the non-lazy random walk to them:
Due to Lemma~\ref{lem:trickling-down-advantage-technical} the following holds for every $1$-level cochain $\stdcochain$:
\[
    \nlupoperator[0]\stdcochain(\stdvertex) = \innerprod{\resviewer[\stdvertex][\stdcochain], \ind{}} = 0
\]
Where the second equality is due to the very definition of $1$-level cochains.
Therefore, while the decomposition that Theorem~\ref{thm:main} guarantees does exist, it is also meaningless as $\nlupoperator[0]\stdcochain = \nlupoperator[0]\stdcochain_0$.
\section{Decomposition of the Random Walk Operators}\label{sec:decomposition-of-the-walk-operators}
We are now ready to present the random walk decomposition theorem based on Theorem~\ref{thm:main}.
Unlike the tricking down theorem, here the assumption we have is only on the expansion of the no-lazy up-down random walk on the vertices.
We will, therefore, be interested in a link viewer that decreases the dimension of the cochain.
Namely, the localization link viewer, defined as follows:
\begin{definition}[Localization]
    Given a simplicial complex $\stdcomplex$ and a cochain $\stdcochain \in \cochainset{k}{\stdcomplex; \R}$ define the localization link viewer in the following way:
    \[
        \forall \stdface \in \stdcomplex: \locviewer[\stdface][\stdcochain] = \stdcochain_{\stdface}
    \]
\end{definition}
As with the trickling down theorem, we will be interested in applying Theorem~\ref{thm:main} to the localization link viewer.
We will start by proving that the localization link viewer is indeed a link viewer that respects the non-lazy up-down random walk.
\begin{lemma}
    The localization link viewer is a link viewer.
\end{lemma}
\begin{proof}
    We will prove the properties point by point:
    \begin{itemize}
        \item $\locviewer[\stdface][\parens{\stdcochain + c \cdot \gencochain}](\genface) = \parens{\stdcochain + c \cdot \gencochain}(\stdface \cup \genface) = \stdcochain(\stdface \cup \genface) + c \cdot \gencochain(\stdface \cup \genface) = \locviewer[\stdface][\stdcochain](\genface) + c \cdot \locviewer[\stdface][\gencochain](\genface)$
        \item $\locviewer[\stdface]\ind{}(\genface) = \ind{}(\stdface \cup \genface) = 1$
        \item For every cochain $\stdcochain$ it holds that $\dim\parens{\locviewer[\stdface][\stdcochain]} = \dim(\stdcochain)-\dim(\stdface)$ therefore for every two cochains $\gencochain, \gencochain'$ and any two $i$-dimensional faces $\genface,\genface'$ it holds that:
        \[
            \dim(\gencochain) - \dim\parens{\locviewer[\genface][\gencochain]} = \dim(\genface) = \dim(\genface') = \dim(\gencochain') - \dim\parens{\locviewer[\genface'][\gencochain']}
        \]
        And $\dimdiff[\locviewer] = 1$.
        \item Let $\genface \subseteq \stdface \in \stdcomplex$ then:
        \[
            \locviewer[\stdface \setminus \genface][\locviewer[\genface][\stdcochain]](\genface')
            = \locviewer[\genface][\stdcochain]\parens{\genface' \cup \parens{\stdface \setminus \genface}}
            = \stdcochain\parens{\genface' \cup \parens{\stdface \setminus \genface} \cup \genface}
            = \stdcochain\parens{\genface' \cup \stdface}
            = \locviewer[\stdface][\stdcochain](\genface')
        \]
        \item For every dimension $i$ and $j=\dim\parens{\viewer[\stdface][\stdcochain]}$ it holds that:
        \begin{multline*}
            \ev{\stdface \in \stdcomplex(i)}{\innerprod{\locviewer[\stdface][\stdcochain], \locviewer[\stdface][\gencochain]}}
             = \sum_{\stdface \in \stdcomplex(i)}{\weight[\stdface]\innerprod{\locviewer[\stdface][\stdcochain], \locviewer[\stdface][\gencochain]}}=\\
             = \sum_{\stdface \in \stdcomplex(i)}{\weight[\stdface]\sum_{\genface \in \stdcomplex_{\stdface}(j)}{\weight[\stdface][\genface]\locviewer[\stdface][\stdcochain](\genface)\locviewer[\stdface][\gencochain](\genface)}} =\\
             = \sum_{\stdface \in \stdcomplex(i)}{\sum_{\genface \in \stdcomplex_{\stdface}(j)}{\frac{\weight[\genface \cup \stdface]}{\binom{i+j+2}{i+1}}\locviewer[\stdface][\stdcochain](\genface)\locviewer[\stdface][\gencochain](\genface)}} =\\
             = \sum_{\stdface \in \stdcomplex(i)}{\sum_{\genface \in \stdcomplex_{\stdface}(j)}{\frac{\weight[\genface \cup \stdface]}{\binom{i+j+2}{i+1}}\stdcochain(\stdface \cup \genface)\gencochain(\stdface \cup \genface)}} = \\
             = \sum_{\genface \in \stdcomplex(i+j+1)}{\weight[\genface]\stdcochain(\genface)\gencochain(\genface)}
             = \innerprod{\stdcochain, \gencochain}
        \end{multline*}
    \end{itemize}
\end{proof}

\begin{lemma}
    The localization link viewer respects the non-lazy random walk operator.
\end{lemma}
\begin{proof}
    Note that, by definition, it holds that
    \[
        \nlupoperator[k]\stdcochain(\stdface)
        = \frac{1}{k+1}\sum_{\substack{\genface \in \stdcomplex(k)\\\stdface \cup \genface \in \stdcomplex(k+1)}}{\weight[\stdface][\genface \setminus \stdface] \stdcochain(\genface)}
    \]
    The Lemma follows from the following calculation:
    \begin{multline*}
        \ev{\stdface \in \stdcomplex(i)}{\innerprod{\nlupoperator[\stdface, j-i-1] \locviewer[\stdface][\stdcochain], \locviewer[\stdface][\gencochain]}}=\\
        = \sum_{\stdface \in \stdcomplex(i)}{\weight[\stdface]\sum_{\genface \in \stdcomplex_{\stdface}(j-i-1)}{\weight[\stdface][\genface]\nlupoperator[\stdface, j-i-1]\locviewer[\stdface][\stdcochain](\genface)\locviewer[\stdface][\gencochain](\genface)}}=\\
        = \frac{1}{j-i}\sum_{\stdface \in \stdcomplex(i)}{\sum_{\genface \in \stdcomplex_{\stdface}(j-i-1)}{\weight[\stdface]\weight[\stdface][\genface]\sum_{\substack{\genface' \in \stdcomplex_{\stdface}(j-i-1)\\\genface\cup\genface' \in \stdcomplex_{\stdface}(j-i)}}{\weight[\stdface \cup \genface][\genface' \setminus \genface]\locviewer[\stdface][\stdcochain](\genface')\locviewer[\stdface][\gencochain](\genface)}}}=\\
        = \frac{1}{j-i}\sum_{\stdface \in \stdcomplex(i)}{\sum_{\genface \in \stdcomplex_{\stdface}(j-i-1)}{\frac{\weight[\stdface \cup \genface]}{\binom{j+1}{i+1}}\sum_{\substack{\genface' \in \stdcomplex_{\stdface}(j-i-1)\\\genface\cup\genface' \in \stdcomplex_{\stdface}(j-i)}}{\weight[\stdface \cup \genface][\genface' \setminus \genface]\locviewer[\stdface][\stdcochain](\genface')\locviewer[\stdface][\gencochain](\genface)}}}=\\
        = \frac{1}{j-i}\frac{1}{\binom{j+1}{i+1}}\sum_{\stdface \in \stdcomplex(i)}{\sum_{\genface \in \stdcomplex_{\stdface}(j-i-1)}{\sum_{\substack{\genface' \in \stdcomplex_{\stdface}(j-i-1)\\\genface\cup\genface' \in \stdcomplex_{\stdface}(j-i)}}{\weight[\stdface \cup \genface]\weight[\stdface \cup \genface][\genface' \setminus \genface]\locviewer[\stdface][\stdcochain](\genface')\locviewer[\stdface][\gencochain](\genface)}}}=\\
        = \frac{1}{j-i}\frac{1}{\binom{j+1}{i+1}}\sum_{\stdface \in \stdcomplex(i)}{\sum_{\genface \in \stdcomplex_{\stdface}(j-i-1)}{\sum_{\substack{\genface' \in \stdcomplex_{\stdface}(j-i-1)\\\genface\cup\genface' \in \stdcomplex_{\stdface}(j-i)}}{\weight[\stdface \cup \genface]\weight[\stdface \cup \genface][\genface' \setminus \genface]\stdcochain(\stdface \cup \genface')\gencochain(\stdface \cup \genface)}}}=\\
        = \frac{1}{j-i}\frac{1}{\binom{j+1}{i+1}}\sum_{\stdface \in \stdcomplex(i)}{\sum_{\substack{\genface \in \stdcomplex(j)\\\stdface \subseteq \genface}}{\sum_{\substack{\genface' \in \stdcomplex(j)\\\genface\cup\genface' \in \stdcomplex_{\stdface}(j+1) \\ \stdface \subseteq \genface'}}{\weight[\genface]\weight[\genface][\genface' \setminus \genface]\stdcochain(\genface')\gencochain(\genface)}}}\underset{(*)}{=}\\
        \underset{(*)}{=} \frac{1}{j+1}\frac{1}{\binom{j}{i+1}}\sum_{\stdface \in \stdcomplex(i)}{\sum_{\substack{\genface \in \stdcomplex(j)\\\stdface \subseteq \genface}}{\sum_{\substack{\genface' \in \stdcomplex(j)\\\genface\cup\genface' \in \stdcomplex_{\stdface}(j+1) \\ \stdface \subseteq \genface'}}{\weight[\genface]\weight[\genface][\genface' \setminus \genface]\stdcochain(\genface')\gencochain(\genface)}}}=\\
        = \frac{1}{j+1}\frac{1}{\binom{j}{i+1}}\sum_{\genface \in \stdcomplex(j)}{\sum_{\substack{\genface' \in \stdcomplex(j)\\\ \genface \cup \genface' \in \stdcomplex(j+1)}}{\sum_{\substack{\stdface \in \stdcomplex(i)\\ \stdface \subseteq \genface \cap \genface'}}{\weight[\genface]\weight[\genface][\genface' \setminus \genface]\stdcochain(\genface')\gencochain(\genface)}}}=\\
        = \frac{1}{j+1}\sum_{\genface \in \stdcomplex(j)}{\weight[\genface] \gencochain(\genface) \sum_{\substack{\genface' \in \stdcomplex(j)\\\ \genface \cup \genface' \in \stdcomplex(j+1)}}{\weight[\genface][\genface' \setminus \genface]\stdcochain(\genface')}}=\\
        = \frac{1}{j+1}\sum_{\genface \in \stdcomplex(j)}{\weight[\genface] \gencochain(\genface) \nlupoperator[\stdface, j] \stdcochain(\genface)}
        =\innerprod{\nlupoperator[\stdface, j-i-1] \stdcochain, \gencochain}
    \end{multline*}
    Where $(*)$ follows from:
    \[
        \frac{1}{j-i}\frac{1}{\binom{j+1}{j-i}}
        = \frac{1}{j-i} \frac{(j-i)!(i+1)!}{(j+1)!}
        = \frac{1}{j+1} \frac{(j-(i+1))!(i+1)!}{j!}
        = \frac{1}{j+1}\frac{1}{\binom{j}{i+1}}
    \]
\end{proof}

\subsection{Gaining the Advantage}\label{subsec:cochains-the-signless-differential-and-the-links}
Before we present the exact Lemma we use as the advantage step we should expand our understanding of $\locviewer$'s level functions.
We will begin by characterising the space of cochains that are orthogonal to the eigenspace of $1$ (i.e.\ the constants).
\begin{lemma}
    Let $\stdcomplex$ be a pure $d$-dimensional simplicial complex and let $\stdcochain \in \cochainset{k}{\stdcochain;\R}$ be a cochain.
    Then:
    \[
        \innerprod{\stdcochain, \mathbbm{1}} = 0 \Leftrightarrow \stdcochain \in \ker{\parens{\sldoperator^*_{-1}\cdots \sldoperator^*_{k-1}}}
    \]
\end{lemma}
\begin{proof}
    Note that, due to Lemma~\ref{lem:multi-sld-operator-adj} it holds that:
    \[
        \innerprod{\stdcochain, \mathbbm{1}} = \ev{\stdface \in \stdcomplex(k)}{\stdcochain(\stdface)} = \sldoperator^*_{-1}\cdots \sldoperator^*_{k-1}(\emptyset)
    \]
    Which proves the lemma.
\end{proof}
Now that we understand the constant part of a cochain we are ready to move on to understanding cochains of a higher level.
\begin{lemma}\label{lem:links-of-cochains-from-high-dimensions}
    Let $\stdcomplex$ be a pure $d$-dimensional simplicial complex, $i$ be a dimension and $\stdcochain \in \cochainset{k}{\stdcomplex;\R}$ be a cochain.
    Then:
    \[
        \forall \stdface \in \stdcomplex(i): \innerprod{\locviewer[\stdface][\stdcochain], \mathbbm{1}}_{\stdface} = 0 \Leftrightarrow \stdcochain \in \ker{\parens{\sldoperator^*_{i}\cdots \sldoperator^*_{k-1}}}
    \]
\end{lemma}
\begin{proof}
    Using Lemma~\ref{lem:multi-sld-operator-adj} we prove that:
    \begin{align*}
        \forall \stdface \in \stdcomplex(i): \innerprod{\locviewer[\stdface][\stdcochain], \mathbbm{1}}_{\stdface}
        & = \sum_{\genface \in \stdcomplex_{\stdface}(0)}{\weight[\stdface][\genface] \locviewer[\stdface][\stdcochain](\genface)}
        = \sldoperator^{*}_{\stdface, i-1} \cdots \sldoperator^{*}_{\stdface, k-2}\stdcochain_{\stdface}(\emptyset)
        = \parens{\sldoperator^{*}_{i} \cdots \sldoperator^{*}_{k-1}\stdcochain}_{\stdface}(\emptyset)\\
        &= \sldoperator^{*}_{i} \cdots \sldoperator^{*}_{k-1}\stdcochain(\stdface)
    \end{align*}
    Which proves the Lemma.
\end{proof}
\begin{corollary}
    It holds that $\stdcochain$ is a proper $k$-dimensional $i$-level cochain iff $\stdcochain \in \im{\parens{\sldoperator_{k-1} \cdots \sldoperator_i}} \cap \ker{\parens{\sldoperator^*_{i-1}\cdots \sldoperator^*_{k-1}}}$.
\end{corollary}
\begin{proof}
    The Corollary holds due the definition of $i$-level cochains and Lemma~\ref{lem:links-of-cochains-from-high-dimensions}.
\end{proof}
Note that cochains that pure $i$-level cochains can be thought of as originating in the $i$-dimensional faces.
For example, for every pure $0$-level cochain there is a $0$-dimensional cochain $\gencochain$ such that $\stdcochain = \sldoperator\cdots\sldoperator\gencochain$.
We also note that any cochain that is not originated in the vertices can be distributed along the links in the sense that they remain orthogonal to the constants when applying the localization link viewer.
We will therefore be interested in the cochains that originated in the vertices (as these are exactly the cochains which the local perspective seems to miss).
Consider the following:
\begin{lemma}\label{lem:from-k-to-0}
Let $\stdcochain$ be a $k$-dimensional proper $0$-level cochain then there exists $\stdcochain^{=0} \in \cochainset{0}{\stdcomplex; \R}$ such that:
\begin{enumerate}
    \item $\sldoperator^*_{-1}\stdcochain^{=0} = 0$
    \item $\norm{\stdcochain}^2 = \norm{\stdcochain^{=0}}^2$
    \item $\norm{\sldoperator^*_{0} \cdots \sldoperator^*_{k-1}\stdcochain}^2 = \norm{\sldoperator_{k-1} \cdots \sldoperator_0 \stdcochain^{=0}}^2$
\end{enumerate}
\end{lemma}
\begin{proof}
    $\stdcochain \in \im{\parens{\sldoperator_{k-1} \cdots \sldoperator_0}}$ therefore let $\gencochain \in \cochainset{0}{\stdcomplex;\R}$ such that $\stdcochain = \sldoperator_{k-1} \cdots \sldoperator_0 \gencochain$.
    Note that $\upoperator[k][0]$ is a self adjoint positive semidefinate operator and thus its square root can be defined.
    $\sqrt{\upoperator[k][0]}$ is defined as the operator whose eigenvectors are the same as $\upoperator[k][0]$ and whose eigenvalue are the positive square root of the eigenvalues of $\upoperator[k][0]$.
    Note that $\sqrt{\upoperator[k][0]}$ is also self adjoint.
    We will now show that $\stdcochain^{=0} = \sqrt{\upoperator[k][0]}\gencochain$ satisfies all of the conditions of the lemma.
    \begin{enumerate}
        \item Note that this condition is equivalent to showing that $\innerprod{\stdcochain^{=0}, \mathbbm{1}} = 0$.
        Also note that $\mathbbm{1}$ is an eigenvector of $\sqrt{\upoperator[k][0]}$ with eigenvalue $1$.
        It is therefore not hard to see that:
        \begin{multline*}
            \innerprod{\stdcochain^{=0}, \mathbbm{1}}
            = \innerprod{\sqrt{\upoperator[k][0]}\gencochain, \sqrt{\upoperator[k][0]} \mathbbm{1}}
            = \innerprod{\upoperator[k][0]\gencochain, \mathbbm{1}}=\\
            = \innerprod{\sldoperator_{k-1}\cdots \sldoperator_{0}\gencochain, \sldoperator^*_{0}\cdots \sldoperator^{*}_{k-1}\mathbbm{1}}
            = \innerprod{\stdcochain, \mathbbm{1}}
            = 0
        \end{multline*}
        \item Consider the following:
        \begin{multline*}
            \norm{\stdcochain^{=0}}^2
            = \innerprod{\stdcochain^{=0}, \stdcochain^{=0}}
            = \innerprod{\sqrt{\upoperator[k][0]}\gencochain, \sqrt{\upoperator[k][0]}\gencochain}
            = \innerprod{\upoperator[k][0]\gencochain, \gencochain}=\\
            = \innerprod{\sldoperator_{k-1} \cdots \sldoperator_{0} \gencochain, \sldoperator_{k-1} \cdots \sldoperator_{0} \gencochain}
            = \innerprod{\stdcochain, \stdcochain}
            = \norm{\stdcochain}^2
        \end{multline*}
        \item To conclude, note that:
        \begin{align*}
            \norm{\sldoperator^*_{0} \cdots \sldoperator^*_{k-1}\stdcochain}^2
            & = \innerprod{\sldoperator^*_{0} \cdots \sldoperator^*_{k-1}\stdcochain, \sldoperator^*_{0} \cdots \sldoperator^*_{k-1}\stdcochain}\\
            & = \innerprod{\sldoperator^*_{0} \cdots \sldoperator^*_{k-1}\sldoperator_{k-1} \cdots \sldoperator_0 \gencochain, \sldoperator^*_{0} \cdots \sldoperator^*_{k-1}\sldoperator_{k-1} \cdots \sldoperator_0 \gencochain}\\
            & = \innerprod{\upoperator[k][0] \gencochain, \upoperator[k][0] \gencochain}\\
            & = \innerprod{\sqrt{\upoperator[k][0]}\sqrt{\upoperator[k][0]} \gencochain, \sqrt{\upoperator[k][0]}\sqrt{\upoperator[k][0]} \gencochain}\\
            & =\innerprod{\sqrt{\upoperator[k][0]}\stdcochain^{=0}, \sqrt{\upoperator[k][0]}\stdcochain^{=0}}\\
            & = \innerprod{\upoperator[k][0]\stdcochain^{=0}, \stdcochain^{=0}}\\
            & = \innerprod{\sldoperator_{k-1} \cdots \sldoperator_0 \stdcochain^{=0}, \sldoperator_{k-1} \cdots \sldoperator_0 \stdcochain^{=0}}\\
            & = \norm{\sldoperator_{k-1} \cdots \sldoperator_0 \stdcochain^{=0}}^2
        \end{align*}
    \end{enumerate}
\end{proof}
We are now ready to present the advantage we use:
\begin{lemma}[The advantage]\label{lem:random-walk-advantage}
    Let $\stdcomplex$ be a $d$ dimensional simplicial complex whose $1$-skeleton is a $\gamma$ spectral expander.
    Also let $\stdcochain \in \levelcochainset{\locviewer,0}{k}{\stdcomplex;\R}$ then:
    \[
        \norm{\sldoperator^*_0 \cdots \sldoperator^*_{k-1}\stdcochain}^2 \le \parens{1-\frac{k}{k+1}\parens{1-\gamma}}\norm{\stdcochain}^2
    \]
\end{lemma}
\begin{proof}
    Let $\stdcochain_0$ be the projection of $\stdcochain$ into $\levelcochainset{\locviewer,0}{k}{\stdcomplex;\R} \cap \parens{\levelcochainset{\locviewer,1}{k}{\stdcomplex;\R}}^\bot$. 
    Due to Lemma~\ref{lem:links-of-cochains-from-high-dimensions} it holds that for every dimension $i$ that $\levelcochainset{\locviewer,i}{k}{\stdcomplex;\R} = \ker{\parens{\sldoperator^*_{i-1}\cdots \sldoperator^*_{k-1}}}$ and therefore $\stdcochain_0 \in \im{\parens{\sldoperator_{k-1} \cdots \sldoperator_0}} \cap \ker{\parens{\sldoperator^*_{-1} \cdots \sldoperator^*_{k-1}}}$.
    We can therefore use Lemma~\ref{lem:from-k-to-0} to find $\stdcochain^{=0}$ such that:
    \begin{enumerate}
        \item $\sldoperator^*_{-1}\stdcochain^{=0} = 0$
        \item $\norm{\stdcochain_0}^2 = \norm{\stdcochain^{=0}}^2$
        \item $\norm{\sldoperator^*_{0} \cdots \sldoperator^*_{k-1}\stdcochain_0}^2 = \norm{\sldoperator_{k-1} \cdots \sldoperator_0 \stdcochain^{=0}}^2$
    \end{enumerate}
    Due to Lemma~\ref{lem:nl-random-walk-and-i-up-operator} it holds that:
    \[
        \upoperator[k][0]
        = \frac{k}{k+1}\nlupoperator[0]+\frac{1}{k+1}I
    \]
    And therefore:
    \begin{align*}
        \norm{\sldoperator_{k-1} \cdots \sldoperator_{0}\stdcochain^{=0}}^2
        & = \innerprod{\sldoperator_{k-1} \cdots \sldoperator_{0}\stdcochain^{=0}, \sldoperator_{k-1} \cdots \sldoperator_{0}\stdcochain^{=0}}
        = \innerprod{\upoperator[k][0]\stdcochain^{=0}, \stdcochain^{=0}}=\\
        &= \innerprod{\parens{\frac{k}{k+1} \nlupoperator[0] - \frac{1}{k+1}I}\stdcochain^{=0}, \stdcochain^{=0}} =\\
        &=  \frac{k}{k+1}\innerprod{\nlupoperator[0]\stdcochain^{=0}, \stdcochain^{=0}} + \frac{1}{k+1}\innerprod{\stdcochain^{=0}, \stdcochain^{=0}} \le \\
        &\le \parens{\frac{k}{k+1}\gamma+\frac{1}{k+1}}\norm{\stdcochain^{=0}}^2 = \\
        &= \parens{\frac{k}{k+1}\gamma-\frac{k}{k+1}+1}\norm{\stdcochain^{=0}}^2=\\
        &= \parens{1-\frac{k}{k+1}\parens{1-\gamma}}\norm{\stdcochain^{=0}}^2
        = \parens{1-\frac{k}{k+1}\parens{1-\gamma}}\norm{\stdcochain_0}^2
    \end{align*}
    And thus:
    \[
        \norm{\sldoperator^*_0 \cdots \sldoperator^*_{k-1}\stdcochain}^2
        = \norm{\sldoperator_{k-1} \cdots \sldoperator_0\stdcochain_0}^2
        \le \parens{1-\frac{k}{k+1}\parens{1-\gamma}}\norm{\stdcochain_0}^2
        \le \parens{1-\frac{k}{k+1}\parens{1-\gamma}}\norm{\stdcochain}^2
    \]
\end{proof}

\subsection{Decomposing the Random Walk Operators}\label{subsec:decomposing-the-random-walk-operators}
Now that we have developed the tools we need, we can move on to strengthening the result of Alev and Lau~\cite{DBLP:journals/corr/abs-2001-02827} by showing a decomposition of the random walk operators.
We will do so by applying Theorem~\ref{thm:main} to the localization link viewer:
\begin{theorem}[Random walk decomposition]\label{thm:walk-operator-decomposition}
    Let $\stdcomplex$ be a $d$-dimensional pure simplicial complex.
    Also, assume that for every face $\stdface$ of dimension smaller than $d-2$ it holds that $\lambda_2\parens{\nlupoperator[\stdface, 0]} \le \lambda_{\stdface}$.
    Denote by $\gamma_{\genface, i} = \max_{\stdface \in \stdcomplex_{\genface}(i)}{\parens{\lambda_{\stdface}}}$.
    For every set of proper level cochains $\stdcochain_i \in \levelcochainset{\locviewer,\hat{i}}{k}{\stdcomplex;\R}$ it holds that:
    \[
        \innerprod{\nlupoperator[k]\sum_{i=0}^{k}{\stdcochain_i},\sum_{i=0}^{k}{\stdcochain_i}} \le \sum_{i=0}^{k}{\parens{1-\frac{1}{k-i+1}\prod_{j=i-1}^{k-1}{(1-\gamma_{j})}}\norm{\stdcochain_i}^2}
    \]
\end{theorem}
\begin{proof}
    We will prove this theorem by applying Theorem~\ref{thm:main} to the $k$-dimensional non-lazy random walk operator.
    We start by noting that the space of $k$-dimensional cochains that are orthogonal to the constants is comprised of exactly $k$ level functions as the space orthogonal to the constants is exactly $\ker\parens{\sldoperator^*_0 \cdots \sldoperator^*_k}$.

    We prove the rest of this theorem using a recursive argument.
    First note that for $k=0$ the claim holds trivially as:
    \[
        \innerprod{\nlupoperator[0]\stdcochain, \stdcochain} \le \lambda_{\stdface} \norm{\stdcochain}^2 = \gamma_{-1} \norm{\stdcochain}^2
    \]
    Assume that for every non-empty face $\stdface$ it holds that:
    \[
        \lambda_{\stdface, i, k} \le 1-\frac{1}{k-i+1}\prod_{j=i-1}^{k-1}{(1-\gamma_{\stdface,j})}
    \]
    Note that for every $i \ge 1$:
    \[
        \lambda_{\emptyset, i, k}
        = \max_{\stdvertex \in \stdcomplex(0)}{\set{\lambda_{\stdvertex, i-1, k-1}}}
        \le \max_{\stdvertex \in \stdcomplex(0)}{\set{1-\frac{1}{k-i+1}\prod_{j=i-2}^{k-2}{(1-\gamma_{\stdface,j})}}}
        \le 1-\frac{1}{k-i+1}\prod_{j=i-1}^{k-1}{(1-\gamma_{\emptyset,j})}
    \]
    Consider the left hand side of the recursive formula:
    \begin{equation}\label{eq:recursive-inequality}
        \lambda_{\emptyset, 1, k}\norm{\stdcochain_0}^2 + \ev{\stdvertex \in \stdcomplex(0)}{(1-\lambda_{\emptyset, 1, k})\norm{\downoperator[(k-1)][(k-1)]\locviewer[\stdvertex][\stdcochain_0]}^2} \le \lambda_{\emptyset, 0, k} \norm{\stdcochain_0}
    \end{equation}
    And note that:
    \begin{align*}
        &\lambda_{\emptyset, 1, k}\norm{\stdcochain_0}^2 + \ev{\stdvertex \in \stdcomplex(0)}{(1-\lambda_{\emptyset, 1, k})\norm{\downoperator[(k-1)][(k-1)]\locviewer[\stdvertex][\stdcochain_0]}^2}\\
        & \qquad \qquad = \lambda_{\emptyset, 1, k}\norm{\stdcochain_0}^2 + \ev{\stdvertex \in \stdcomplex(0)}{(1-\lambda_{\emptyset, 1, k})\norm{\sldoperator^*_{-1} \cdots \sldoperator^*_{k-2} \locviewer[\stdvertex][\stdcochain_0]}^2}\\
        & \qquad \qquad = \lambda_{\emptyset, 1, k}\norm{\stdcochain_0}^2 + \ev{\stdvertex \in \stdcomplex(0)}{(1-\lambda_{\emptyset, 1, k})\norm{\locviewer[\stdvertex][\sldoperator^*_0 \cdots \sldoperator^*_{k-1}\stdcochain_0]}^2}\\
        & \qquad \qquad = \lambda_{\emptyset, 1, k}\norm{\stdcochain_0}^2 + (1-\lambda_{\emptyset, 1, k})\norm{\sldoperator^*_0 \cdots \sldoperator^*_{k-1}\stdcochain_0}^2
    \end{align*}
    Consider, again, the left hand side of inequality~\ref{eq:recursive-inequality} and note that due to Lemma~\ref{lem:random-walk-advantage} it suffices to solve the following:
    \begin{align*}
        \lambda_{\emptyset, 1, k}\norm{\stdcochain^{=0}}^2 + (1-\lambda_{\emptyset, 1, k})\parens{1-\frac{k}{k+1}\parens{1-\gamma_{-1}}}\norm{\stdcochain^{=0}}^2 &\le \lambda_{\emptyset, 0, k} \norm{\stdcochain^{=0}}^2\\
        \lambda_{\emptyset, 1, k} + (1-\lambda_{\emptyset, 1, k})\parens{1-\frac{k}{k+1}\parens{1-\gamma_{-1}}} &\le \lambda_{\emptyset, 0, k}\\
        \frac{k}{k+1}\parens{1-\gamma_{-1}} \lambda_{\emptyset, 1, k} + \parens{1-\frac{k}{k+1}\parens{1-\gamma_{-1}}} &\le \lambda_{\emptyset, 0, k}
    \end{align*}
    Consider the following:
    \begin{multline*}
        \frac{k}{k+1}\parens{1-\gamma_{-1}} \lambda_{\emptyset, 1, k} + \parens{1-\frac{k}{k+1}\parens{1-\gamma_{-1}}} \le \\
        \le \frac{k}{k+1}\parens{1-\gamma_{-1}} \parens{1-\frac{1}{k}\prod_{j=0}^{k-1}{(1-\gamma_{\emptyset,j})}} + 1-\frac{k}{k+1}\parens{1-\gamma_{-1}} = \\
        = \frac{k}{k+1}\parens{1-\gamma_{-1}} -\frac{1}{k+1}\prod_{j=-1}^{k-1}{(1-\gamma_{\emptyset,j})} + 1-\frac{k}{k+1}\parens{1-\gamma_{-1}} = \\
        = 1 - \frac{1}{k+1}\prod_{j=-1}^{k-1}{(1-\gamma_{\emptyset,j})}
    \end{multline*}
    Thus if we set:
    \[
        \lambda_{\emptyset,0,k} = 1 - \frac{1}{k+1}\prod_{j=-1}^{k-1}{(1-\gamma_{\emptyset,j})}
    \]
    We get that for every face $\stdface$ and dimensions $i,k$:
    \[
        \lambda_{\stdface,i,k} \le 1 - \frac{1}{k-i+1}\prod_{j=i-1}^{k-1}{(1-\gamma_{\stdface,j})}
    \]
    Applying Theorem~\ref{thm:main} proves the decomposition.
\end{proof}
Note that the decomposition presented in Theorem~\ref{thm:walk-operator-decomposition} is the first decomposition theorem to offer a proper (i.e. not approximate) decomposition of the $k$-dimensional random walk whose components are orthogonal to each other.

\subsection{Examples of Cochains of a High Level}\label{subsec:examples-of-cochains-of-high-level}
In this section we will show two examples of cochains of dimension $k$ are $k$-level cochains.
Specifically we will show two examples of structures that have this property - the first is minimal cochains and the second is perfectly balanced cochains.

\subsubsection{Minimal Cochains}\label{subsubsec:minimal-cochains}
Minimal cochains appear naturally in a topological definition of expansion and were widely studied over $\mathbbm{F}_2$ (as most usages of topological expanders are of topological expanders over $\mathbbm{F}_2$).
We will present a definition of an analogue of these cochains over the real numbers and show that minimal $k$-cochains over $\R$ are of level $k$.

We begin by noting that the topological definition of expansion uses \emph{oriented} complexes and thus we have to modify our definition of a cochain to match an oriented complex:
\begin{definition}[Oriented complex]
    We define of an oriented simplicial complex $\stdcomplex$ as a simplicial complex that has an underlying orientation on the vertices.
    We also define the set $\stdcomplex_{ord}$ to be the set of all possible orientations of the faces of $\stdcomplex$.
\end{definition}
\begin{definition}[Cochain over oriented complex]
    A cochain $\stdcochain \in \cochainset{i}{\stdcomplex_{ord};\R}$ is a function $\stdcochain: \stdcomplex_{ord} \rightarrow \R$ such that for every permutation $\rho$ and every face $\stdface \in \stdcomplex_{ord}$ it holds that:
    \[
        \stdcochain(\stdface_{\rho(0)},\cdots,\stdface_{\rho(i)}) = \sign(\rho) \stdcochain(\stdface_{0},\cdots,\stdface_{i})
    \]
    For the rest of this section we will assume that the complex has some underlying orientation and that faces are given in that underlying distribution.
    For that reason, unless otherwise stated, we will ignore the orientation of the complex.
\end{definition}
Another key definition in topological notions of high dimensional expanders is the coboundary operator.
In the $0$th dimension over $\mathbbm{F}_2$ the coboundary operator can be thought of as marking all the edges that leave a set of vertices.
\begin{definition}[Coboundary operator]
    Define the coboundary operator $\coboundaryoperator_i: \cochainset{i}{\stdcomplex_{ord};\R} \rightarrow \cochainset{i+1}{\stdcomplex_{ord};\R}$ to be:
    \[
        \coboundaryoperator_i\stdcochain(\stdface_0,\cdots,\stdface_{i+1}) = \sum_{j=0}^{i+1}{(-1)^j \stdcochain(\stdface_0,\cdots,\stdface_{j-1}, \stdface_{j+1},\cdots, \stdface_{i+1})}
    \]
    Note that the dimension of the coboundary operator is always clear from context and thus we will often denote the coboundary operator simply by $\coboundaryoperator$.
\end{definition}
\begin{definition}[Coboundary]
    We define the set of coboundaries to be:
    \[
        \coboundaryset{i}{\stdcomplex;\R} = \image(\coboundaryoperator_{i-1})
    \]
\end{definition}
We are now ready to define a minimal cochain - a family of cochains that are all $k$-level cochains.
Note that these are one of the central objects of study when discussing topological definitions of high dimensioanl expanders (specifically, the expansion of minimal cochains under applications of the coboundary operator is the very definition of that notion of expansion\footnote{Specifically, this holds true for a notion called coboundary expansion. There are other notions of topological expansion such as cosystolic expansion in which minimality is taken with regards to a larger family of cochains called ``the cocycles''.}).
For a more in-depth discussion of the topological notions of expansion and their relation to minimal cochains see, for example \cite{DBLP:conf/focs/KaufmanKL14, DBLP:conf/stoc/EvraK16, DBLP:conf/isaac/KaufmanM21, DBLP:conf/approx/KaufmanM22, DBLP:conf/approx/KaufmanO22}.
\begin{definition}[Minimal cochain]
    A cochain $\stdcochain \in \cochainset{i}{\stdcomplex;\R}$ is minimal if:
    \[
        \stdcochain = \argmin_{\gencochain \in \coboundaryset{i}{\stdcomplex;\R}}{\norm{\stdcochain + \gencochain}}
    \]
\end{definition}
In order to show that that every minimal cochain is indeed a $k$-level cochain we will first consider a weaker notion of minimality called local-minimality.
\begin{definition}[Locally minimal cochain]
    A cochain $\stdcochain \in \cochainset{i}{\stdcomplex;\R}$ is $\locviewer$-$i$-locally minimal if for every $\stdface \in \stdcomplex(i)$ it holds that $\locviewer[\stdface][\stdcochain]$ is minimal.
\end{definition}
We will show that every $k$-cochain that is $(k-1)$-locally minimal is also a $k$-level cochain.
We will then follow this by showing that any minimal cochain is $(k-1)$-locally minimal and thus prove the claim.

\begin{lemma}\label{lem:constants-and-minimality}
    Let $\stdcomplex$ be a simplicial complex and let $\stdcochain \in \cochainset{i}{\stdcomplex; \R}$ be a cochain.
    It holds that $\norm{\stdcochain-a}^2 < \norm{\stdcochain}^2$ if one of the following conditions hold:
    \begin{itemize}
        \item $a \le 0$ and $a > 2 \ev{\stdface \in \stdcomplex(i)}{\stdcochain(\stdface)}$.
        \item $a \ge 0$ and $a < 2 \ev{\stdface \in \stdcomplex(i)}{\stdcochain(\stdface)}$
    \end{itemize}
\end{lemma}
\begin{proof}
    Note that the following hold:
    \begin{multline*}
        \norm{\stdcochain - a}^2 < \norm{\stdcochain}^2
        \Leftrightarrow \sum_{\stdface \in \stdcomplex(i)}{\weight[\stdface]\parens{\stdcochain(\stdface)-a}^2} < \sum_{\stdface \in \stdcomplex(i)}{\weight[\stdface]\stdcochain(\stdface)^2} \Leftrightarrow \\
        \Leftrightarrow \sum_{\stdface \in \stdcomplex(i)}{\weight[\stdface]\stdcochain(\stdface)^2} -2a\sum_{\stdface \in \stdcomplex(i)}{\weight[\stdface]\stdcochain(\stdface)} + a^2 < \sum_{\stdface \in \stdcomplex(i)}{\weight[\stdface]\stdcochain(\stdface)^2} \Leftrightarrow \\
        \Leftrightarrow a^2 < 2a\sum_{\stdface \in \stdcomplex(i)}{\weight[\stdface]\stdcochain(\stdface)}
    \end{multline*}

    If $a \ge 0$ then:
    \[
        a^2 < 2a\sum_{\stdface \in \stdcomplex(i)}{\weight[\stdface]\stdcochain(\stdface)}
        \Leftrightarrow a < 2\sum_{\stdface \in \stdcomplex(i)}{\weight[\stdface]\stdcochain(\stdface)}
    \]

    In addition, if $a \le 0$ then:
    \[
        a^2 < 2a\sum_{\stdface \in \stdcomplex(i)}{\weight[\stdface]\stdcochain(\stdface)}
        \Leftrightarrow a > 2\sum_{\stdface \in \stdcomplex(i)}{\weight[\stdface]\stdcochain(\stdface)}
    \]
\end{proof}
\begin{corollary}
    Let $\stdcochain \in \cochainset{0}{\stdcomplex; \R}$ be a cochain.
    If $\stdcochain$ is minimal then $\ev{\stdface \in \stdcomplex(0)}{\stdcochain(\stdface)} = 0$.
\end{corollary}
\begin{proof}
    Note that the $0$-dimensioanl coboundaries are the constant functions and thus if $\ev{\stdface \in \stdcomplex(0)}{\stdcochain(\stdface)} \ne 0$ then note that $a = \ev{\stdface \in \stdcomplex(0)}{\stdcochain(\stdface)}$ satisfies the conditions of Lemma~\ref{lem:constants-and-minimality} and thus contradict the minimality of $\stdcochain$.
\end{proof}
\begin{corollary}\label{cor:locally-minimal-cochains-are-k-level}
    $(k-1)$-locally minimal cochains are $k$-level cochains.
\end{corollary}
\begin{proof}
    Let $\stdcochain$ be a $(k-1)$-locally minimal cochain and let $\stdface \in \stdcomplex(k-1)$.
    Due to the local minimality of $\stdcochain$ it holds that $\locviewer[\stdface][\stdcochain]$ is minimal in $\stdcomplex_{\stdface}$ i.e. for every $\gencochain \in \coboundaryset{0}{\stdcomplex_{\stdface};\R}$ it holds that $\norm{\stdcochain} \le \norm{\stdcochain + \gencochain}$.
    Note that the $0$-coboundaries of $\stdcomplex_{\stdface}$ over $\R$ are the constant functions and therefore, due to Lemma~\ref{lem:constants-and-minimality}, it holds that $\ev{\stdface \in \stdcomplex(i)}{\stdcochain(\stdface)} = 0$ (as if it is not one can pick a coboundary that minimizes the norm of $\locviewer[\stdface][\stdcochain]$).
\end{proof}
\begin{lemma}\label{lem:minimal-cochains-are-locally-minimal}
    Minimal cochains are $(k-1)$-locally minimal.
\end{lemma}
\begin{proof}
    Let $\stdcochain \in \cochainset{k}{\stdcomplex;\R}$ be a minimal cochain.
    Suppose that it is not $(k-1)$-locally minimal, therefore there exists a face $\stdface \in \stdcomplex(k-1)$ such that $\locviewer[\stdface][\stdcochain]$ is not minimal.
    Note that $\locviewer[\stdface][\stdcochain] \in \cochainset{0}{\stdcomplex_{\stdface};\R}$ and therefore, due to the non-minimality of $\locviewer[\stdface][\stdcochain]$ in $\stdcomplex_{\stdface}$, there exists $a \in \R$ such that $\norm{\stdcochain-a}^2_{\stdface} < \norm{\stdcochain}^2_{\stdface}$ (as $0$-coboundaries over $\R$ are the constant functions).
    Consider the following $(k-1)$-cochain:
    \[
        \gencochain(\genface) = \begin{cases}
            a & \genface = \stdface \\
            0 & \text{Otherwise}
        \end{cases}
    \]
    The following holds:
    \begin{align*}
        \norm{\stdcochain - \coboundaryoperator\gencochain}^2
        &= \sum_{\genface \in \stdcomplex(k)}{\weight[\genface]\parens{\stdcochain(\genface) - \coboundaryoperator\gencochain(\genface)}^2} = \\
        &= \sum_{\substack{\genface \in \stdcomplex(k) \\ \stdface \nsubseteq \genface}}{\weight[\genface]\parens{\stdcochain(\genface) - \coboundaryoperator\gencochain(\genface)}^2} + \sum_{\substack{\genface \in \stdcomplex(k) \\ \stdface \subseteq \genface}}{\weight[\genface]\parens{\stdcochain(\genface) - \coboundaryoperator\gencochain(\genface)}^2} = \\
        &= \sum_{\substack{\genface \in \stdcomplex(k) \\ \stdface \nsubseteq \genface}}{\weight[\genface]\parens{\stdcochain(\genface)-0}^2} + \sum_{\stdvertex \in \stdcomplex_{\stdface}(0)}{\weight[\genface]\parens{\stdcochain(\stdvertex\stdface) - \coboundaryoperator\gencochain(\stdvertex\stdface)}^2} = \\
        &= \sum_{\substack{\genface \in \stdcomplex(k) \\ \stdface \nsubseteq \genface}}{\weight[\genface]\stdcochain(\genface)}^2 + \sum_{\stdvertex \in \stdcomplex_{\stdface}(0)}{\weight[\stdvertex\stdface]\parens{\stdcochain(\stdvertex\stdface) - a}^2} = \\
        &= \sum_{\substack{\genface \in \stdcomplex(k) \\ \stdface \nsubseteq \genface}}{\weight[\genface]\stdcochain(\genface)^2} + \binom{k+1}{k}\weight[\stdface]\sum_{\stdvertex \in \stdcomplex_{\stdface}(0)}{\weight[\stdface][\stdvertex]\parens{\stdcochain(\stdvertex\stdface) - a}^2} = \\
        &= \sum_{\substack{\genface \in \stdcomplex(k) \\ \stdface \nsubseteq \genface}}{\weight[\genface]\stdcochain(\genface)^2} + \binom{k+1}{k}\weight[\stdface]\norm{\stdcochain-a}^2_{\stdface} < \\
        &< \sum_{\substack{\genface \in \stdcomplex(k) \\ \stdface \nsubseteq \genface}}{\weight[\genface]\stdcochain(\genface)^2} + \binom{k+1}{k}\weight[\stdface]\norm{\stdcochain}^2_{\stdface} = \\
        &= \sum_{\substack{\genface \in \stdcomplex(k) \\ \stdface \nsubseteq \genface}}{\weight[\genface]\stdcochain(\genface)^2} + \binom{k+1}{k}\weight[\stdface]\sum_{\stdvertex \in \stdcomplex_{\stdface}(0)}{\weight[\stdface][\stdvertex]\parens{\stdcochain(\stdvertex\stdface)}^2} = \\
        &= \sum_{\substack{\genface \in \stdcomplex(k) \\ \stdface \nsubseteq \genface}}{\weight[\genface]\stdcochain(\genface)^2} + \sum_{\substack{\genface \in \stdcomplex(k) \\ \stdface \subseteq \genface}}{\weight[\genface]\stdcochain(\genface)^2}
        = \sum_{\genface \in \stdcomplex(k)}{\weight[\genface]\stdcochain(\genface)^2}
        = \norm{\stdcochain}^2
    \end{align*}
    Which contradicts the minimality of $\stdcochain$.
\end{proof}
\begin{corollary}
    Minimal $k$-cochains are $k$-levels cochains.
\end{corollary}
\begin{proof}
    This follows directly from Corollary~\ref{cor:locally-minimal-cochains-are-k-level} and Lemma~\ref{lem:minimal-cochains-are-locally-minimal}.
\end{proof}

\subsubsection{Perfectly Balanced Cochains}\label{subsubsec:perfectly-balanced-cochains}
We will now show a different family of cochains that are $k$-level cochains.
Specifically we define the notion of perfectly balanced cochains.
This will be a purely combinatorial definition and we will show that cochains that satisfy this combinatorial definition are $k$-level cochains (up to an additive constant).
\begin{definition}[Perfectly balanced set of faces]
    A set of faces $S \subseteq \stdcomplex(k)$ is perfectly balanced over links of dimension $i$ if:
    \[
        \forall \stdface \in \stdcomplex(i): \sum_{\genface \in S}{\weight[\genface]} = \sum_{\substack{\genface \in S \\ \stdface \subseteq \genface}}{\weight[\stdface][\genface \setminus \stdface]}
    \]
\end{definition}
\begin{definition}[Perfectly balanced cochain]
    A cochain $\stdcochain \in \cochainset{k}{\stdcomplex;\R}$ is perfectly balanced over faces of dimension $i$ if there exists a perfectly balanced set of faces $S \subseteq \stdcomplex(k)$ that is perfectly balanced over faces of dimension $i$ such that:
    \[
        \stdcochain(\stdface) = \begin{cases}
            1 & \stdface \in S \\
            0 & \stdface \notin S
        \end{cases}
    \]
\end{definition}
\begin{lemma}
    Let $\stdcochain \in \cochainset{k}{\stdcomplex;\R}$ be perfectly balanced over faces of dimension $i$ then $\stdcochain - \ev{}{\stdcochain}$ is an $i$-level cochain.
\end{lemma}
\begin{proof}
    The Lemma follows directly from the definition:
    \begin{multline*}
        \forall \stdface \in \stdcomplex(i): \ev{\genface \in \stdcomplex_{\stdface}(k-i-1)}{\locviewer[\stdface][\parens{\stdcochain - \ev{}{\stdcochain}}](\genface)} = \\
        = \ev{\genface \in \stdcomplex_{\stdface}(k-i-1)}{\locviewer[\stdface][\stdcochain](\genface) - \ev{}{\stdcochain}}
        = \ev{\genface \in \stdcomplex(k)}{\stdcochain(\genface)} - \ev{}{\stdcochain}
        = 0
    \end{multline*}
\end{proof}

    \bibliographystyle{alpha}
    \bibliography{main}

\newcommand{\etalchar}[1]{$^{#1}$}
\begin{thebibliography}{DHK{\etalchar{+}}19}

\bibitem[AJK{\etalchar{+}}22]{anari2022entropic}
Nima Anari, Vishesh Jain, Frederic Koehler, Huy~Tuan Pham, and Thuy-Duong
  Vuong.
\newblock Entropic independence: optimal mixing of down-up random walks.
\newblock In {\em Proceedings of the 54th Annual ACM SIGACT Symposium on Theory
  of Computing}, pages 1418--1430, 2022.

\bibitem[AL20]{DBLP:journals/corr/abs-2001-02827}
Vedat~Levi Alev and Lap~Chi Lau.
\newblock Improved analysis of higher order random walks and applications.
\newblock {\em CoRR}, abs/2001.02827, 2020.

\bibitem[ALG20]{anari2021spectral}
Nima Anari, Kuikui Liu, and Shayan~Oveis Gharan.
\newblock Spectral independence in high-dimensional expanders and applications
  to the hardcore model.
\newblock In Sandy Irani, editor, {\em 61st {IEEE} Annual Symposium on
  Foundations of Computer Science, {FOCS} 2020, Durham, NC, USA, November
  16-19, 2020}, pages 1319--1330. {IEEE}, 2020.

\bibitem[ALGV18]{https://doi.org/10.48550/arxiv.1811.01816}
Nima Anari, Kuikui Liu, Shayan~Oveis Gharan, and Cynthia Vinzant.
\newblock Log-concave polynomials ii: High-dimensional walks and an fpras for
  counting bases of a matroid, 2018.

\bibitem[BHKL20]{eigenstripping-pseudorandomness-and-unique-games}
Mitali Bafna, Max Hopkins, Tali Kaufman, and Shachar Lovett.
\newblock High dimensional expanders: Eigenstripping, pseudorandomness, and
  unique games, 2020.

\bibitem[BHKL21]{https://doi.org/10.48550/arxiv.2111.09444}
Mitali Bafna, Max Hopkins, Tali Kaufman, and Shachar Lovett.
\newblock Hypercontractivity on high dimensional expanders: a local-to-global
  approach for higher moments, 2021.

\bibitem[CLV21]{chen2021optimal}
Zongchen Chen, Kuikui Liu, and Eric Vigoda.
\newblock Optimal mixing of glauber dynamics: Entropy factorization via
  high-dimensional expansion.
\newblock In {\em Proceedings of the 53rd Annual ACM SIGACT Symposium on Theory
  of Computing}, pages 1537--1550, 2021.

\bibitem[DDFH18]{DBLP:conf/approx/DiksteinDFH18}
Yotam Dikstein, Irit Dinur, Yuval Filmus, and Prahladh Harsha.
\newblock Boolean function analysis on high-dimensional expanders.
\newblock In Eric Blais, Klaus Jansen, Jos{\'{e}} D.~P. Rolim, and David
  Steurer, editors, {\em Approximation, Randomization, and Combinatorial
  Optimization. Algorithms and Techniques, {APPROX/RANDOM} 2018, August 20-22,
  2018 - Princeton, NJ, {USA}}, volume 116 of {\em LIPIcs}, pages 38:1--38:20.
  Schloss Dagstuhl - Leibniz-Zentrum f{\"{u}}r Informatik, 2018.

\bibitem[DHK{\etalchar{+}}19]{DBLP:conf/soda/DinurHKNT19}
Irit Dinur, Prahladh Harsha, Tali Kaufman, Inbal~Livni Navon, and Amnon
  Ta{-}Shma.
\newblock List decoding with double samplers.
\newblock In Timothy~M. Chan, editor, {\em Proceedings of the Thirtieth Annual
  {ACM-SIAM} Symposium on Discrete Algorithms, {SODA} 2019, San Diego,
  California, USA, January 6-9, 2019}, pages 2134--2153. {SIAM}, 2019.

\bibitem[DK17]{DBLP:conf/focs/DinurK17}
Irit Dinur and Tali Kaufman.
\newblock High dimensional expanders imply agreement expanders.
\newblock In Chris Umans, editor, {\em 58th {IEEE} Annual Symposium on
  Foundations of Computer Science, {FOCS} 2017, Berkeley, CA, USA, October
  15-17, 2017}, pages 974--985. {IEEE} Computer Society, 2017.

\bibitem[EK16]{DBLP:conf/stoc/EvraK16}
Shai Evra and Tali Kaufman.
\newblock Bounded degree cosystolic expanders of every dimension.
\newblock In Daniel Wichs and Yishay Mansour, editors, {\em Proceedings of the
  48th Annual {ACM} {SIGACT} Symposium on Theory of Computing, {STOC} 2016,
  Cambridge, MA, USA, June 18-21, 2016}, pages 36--48. {ACM}, 2016.

\bibitem[GLL22]{https://doi.org/10.48550/arxiv.2111.09375}
Tom Gur, Noam Lifshitz, and Siqi Liu.
\newblock Hypercontractivity on high dimensional expanders.
\newblock In Stefano Leonardi and Anupam Gupta, editors, {\em {STOC} '22: 54th
  Annual {ACM} {SIGACT} Symposium on Theory of Computing, Rome, Italy, June 20
  - 24, 2022}, pages 176--184. {ACM}, 2022.

\bibitem[KKL14]{DBLP:conf/focs/KaufmanKL14}
Tali Kaufman, David Kazhdan, and Alexander Lubotzky.
\newblock Ramanujan complexes and bounded degree topological expanders.
\newblock In {\em 55th {IEEE} Annual Symposium on Foundations of Computer
  Science, {FOCS} 2014, Philadelphia, PA, USA, October 18-21, 2014}, pages
  484--493. {IEEE} Computer Society, 2014.

\bibitem[KM17]{DBLP:conf/innovations/KaufmanM17}
Tali Kaufman and David Mass.
\newblock High dimensional random walks and colorful expansion.
\newblock In Christos~H. Papadimitriou, editor, {\em 8th Innovations in
  Theoretical Computer Science Conference, {ITCS} 2017, January 9-11, 2017,
  Berkeley, CA, {USA}}, volume~67 of {\em LIPIcs}, pages 4:1--4:27. Schloss
  Dagstuhl - Leibniz-Zentrum f{\"{u}}r Informatik, 2017.

\bibitem[KM20]{DBLP:conf/innovations/KaufmanM20}
Tali Kaufman and David Mass.
\newblock Local-to-global agreement expansion via the variance method.
\newblock In Thomas Vidick, editor, {\em 11th Innovations in Theoretical
  Computer Science Conference, {ITCS} 2020, January 12-14, 2020, Seattle,
  Washington, {USA}}, volume 151 of {\em LIPIcs}, pages 74:1--74:14. Schloss
  Dagstuhl - Leibniz-Zentrum f{\"{u}}r Informatik, 2020.

\bibitem[KM21]{DBLP:conf/isaac/KaufmanM21}
Tali Kaufman and David Mass.
\newblock Unique-neighbor-like expansion and group-independent cosystolic
  expansion.
\newblock In Hee{-}Kap Ahn and Kunihiko Sadakane, editors, {\em 32nd
  International Symposium on Algorithms and Computation, {ISAAC} 2021, December
  6-8, 2021, Fukuoka, Japan}, volume 212 of {\em LIPIcs}, pages 56:1--56:17.
  Schloss Dagstuhl - Leibniz-Zentrum f{\"{u}}r Informatik, 2021.

\bibitem[KM22]{DBLP:conf/approx/KaufmanM22}
Tali Kaufman and David Mass.
\newblock Double balanced sets in high dimensional expanders.
\newblock In Amit Chakrabarti and Chaitanya Swamy, editors, {\em Approximation,
  Randomization, and Combinatorial Optimization. Algorithms and Techniques,
  {APPROX/RANDOM} 2022, September 19-21, 2022, University of Illinois,
  Urbana-Champaign, {USA} (Virtual Conference)}, volume 245 of {\em LIPIcs},
  pages 3:1--3:17. Schloss Dagstuhl - Leibniz-Zentrum f{\"{u}}r Informatik,
  2022.

\bibitem[KO18]{DBLP:conf/approx/KaufmanO18}
Tali Kaufman and Izhar Oppenheim.
\newblock High order random walks: Beyond spectral gap.
\newblock In Eric Blais, Klaus Jansen, Jos{\'{e}} D.~P. Rolim, and David
  Steurer, editors, {\em Approximation, Randomization, and Combinatorial
  Optimization. Algorithms and Techniques, {APPROX/RANDOM} 2018, August 20-22,
  2018 - Princeton, NJ, {USA}}, volume 116 of {\em LIPIcs}, pages 47:1--47:17.
  Schloss Dagstuhl - Leibniz-Zentrum f{\"{u}}r Informatik, 2018.

\bibitem[KO22]{DBLP:conf/approx/KaufmanO22}
Tali Kaufman and Izhar Oppenheim.
\newblock High dimensional expansion implies amplified local testability.
\newblock In Amit Chakrabarti and Chaitanya Swamy, editors, {\em Approximation,
  Randomization, and Combinatorial Optimization. Algorithms and Techniques,
  {APPROX/RANDOM} 2022, September 19-21, 2022, University of Illinois,
  Urbana-Champaign, {USA} (Virtual Conference)}, volume 245 of {\em LIPIcs},
  pages 5:1--5:10. Schloss Dagstuhl - Leibniz-Zentrum f{\"{u}}r Informatik,
  2022.

\bibitem[Opp17]{https://doi.org/10.48550/arxiv.1709.04431}
Izhar Oppenheim.
\newblock Local spectral expansion approach to high dimensional expanders part
  i: Descent of spectral gaps, 2017.

\end{thebibliography}

\end{document}